\documentclass[11pt]{article}
\usepackage[]{geometry}
\usepackage{graphicx}
\usepackage{caption}
\usepackage{subcaption}
\usepackage{hyperref}
\usepackage{amsmath}
\usepackage{amsthm}
\usepackage{amssymb}
\usepackage{authblk}
\usepackage{amsmath}
\usepackage{bm}
\usepackage{xcolor}
\usepackage{float}
\usepackage{booktabs}

\usepackage{yhmath}
\usepackage{mathdots}
\usepackage{MnSymbol}

\everymath{\displaystyle}
\allowdisplaybreaks

\newtheorem{theorem}{Theorem}[section]
\newtheorem{definition}{Definition}[section]
\newtheorem{proposition}{Proposition}[section]
\newtheorem{remark}{Remark}[section]
\newtheorem{corollary}{Corollary}[section]
\newtheorem{lemma}{Lemma}[section]

\providecommand{\keywords}[1]
{ \small \textbf{\textit{Keywords---}}#1 }

\title{IGA Laplace Eigenfrequencies Distributions and Estimations: Impact of Reparametrization on Eigenfrequency Behavior}
\author[$\;$]{}

\affil[$\;$]{ {\large Noureddine LAMSAHEL}}
\affil[$\;$]{ {\scriptsize  Mohammed VI Polytechnic University, The UM6P Vanguard Center, Benguerir 43150, Lot 660, Hay Moulay Rachid, Morocco.}}
\affil[$\;$]{ {\scriptsize  University of Littoral C\^ote d'Opale, Laboratory of Pure and Applied Mathematics, 50 Rue F. Buisson, 62228 Calais-Cedex, France.}}
\affil[$\;$]{ {\scriptsize Emails: noureddine.lamsahel@um6p.ma, noureddine.lamsahel@etu.univ-littoral.fr  }}
\affil[$\;$]{ {\large }}
\affil[$\;$]{ {\large Abdeladim EL AKRI}}
\affil[$\;$]{ {\scriptsize   Mohammed VI Polytechnic University, The UM6P Vanguard Center, Benguerir 43150, Lot 660, Hay Moulay Rachid, Morocco.}}
\affil[$\;$]{ {\scriptsize Email: Abdeladim.ElAkri@um6p.ma}}
\affil[$\;$]{ {\large }}
\affil[$\;$]{ {\large Ahmed RATNANI}}
\affil[$\;$]{ {\scriptsize  Mohammed VI Polytechnic University, The UM6P Vanguard Center, Benguerir 43150, Lot 660, Hay Moulay Rachid, Morocco.}}
\affil[$\;$]{ {\scriptsize Email: Ahmed.RATNANI@um6p.ma}}
\affil[$\;$]{ {\large }}

\date{}

\begin{document}

\maketitle

\begin{abstract}
This work addresses the Galerkin isogeometric discretization of the one-dimensional Laplace eigenvalue problem subject to homogeneous Dirichlet boundary conditions on a bounded interval. We employ GLT theory to analyze the behavior of the eigenfrequencies when a reparametrization is applied to the computational domain. Under suitable assumptions on the reparametrization transformation, we prove that a structured pattern emerges in the distribution of eigenfrequencies when the problem is reframed through GLT-symbol analysis. Additionally, we establish results that refine and extend those of \cite{bianchi2021analysis}, including a uniform discrete Weyl's law. Furthermore, we derive several eigenfrequency estimates by establishing that the symbol exhibits asymptotically linear behavior near zero. 
\end{abstract}

\noindent\keywords{Laplace operator, Isogeometric Galerkin discretization, Reparametrization,  GLT theory, eigenfrequencies, Behavior,  Spectral symbol.}

\section{Introduction}\label{sec1}
Isogeometric Analysis (IGA), introduced in \cite{hughes2005isogeometric}, is a powerful paradigm for analyzing and approximating partial differential equations \cite{cottrell2009isogeometric}. The main purpose of IGA is to bridge computer-aided design (CAD) and finite element analysis (FEA), see \cite{cottrell2009isogeometric,hughes2005isogeometric}. This connection gives IGA an advantage over classical (FEA). In particular, the high smoothness of B-spline and NURBS basis functions enables higher accuracy per degree of freedom \cite{bressan2019approximation,sande2019sharp,sande2020explicit} and provides a better description of the spectrum of the involved operator compared to classical (FEA), see \cite{cottrell2006isogeometric,hughes2014finite,hughes2008duality}.

Several authors have considered the analysis of the IGA Laplace operator problem as a foundational study for solving and investigating IGA problems; see, for instance, \cite{cottrell2006isogeometric,ekstrom2018eigenvalues,garoni2019symbol,manni2022application,reali2006isogeometric}. From the works \cite{sande2019sharp,sande2020explicit}, the authors deduce that B-splines of maximal smoothness on uniform grids are preferable for addressing the Laplace eigenvalue problem. However, despite this accurate description and estimation of the spectrum, there are always a few eigenvalues, called outliers \cite{cottrell2006isogeometric}, that are poorly approximated, with their corresponding values being significantly larger than the exact ones. In \cite{manni2022application}, the authors overcame this problem by constructing subspaces of the full IGA spaces subject to the three standard boundary conditions and proving that these spaces filter out the outliers while yielding the optimal approximation error. Compared to other outlier-free spaces, such as reduced spaces \cite{hiemstra2021removal}, the authors in \cite{manni2022application} provide a mathematical confirmation that guarantees the outlier-free property of these spaces.

Although optimal spline subspaces \cite{manni2022application} are the best spaces for addressing the Laplace eigenvalue problem, the construction of these spaces depends on selective uniform grids that vary according to the chosen boundary conditions and the parity of the interpolation degree. However, some applications require the use of non-uniform grids. In particular, for the uniform boundary observability of the wave equation, it is widely acknowledged that almost all classical discretizations of the $1d$ wave equation using uniform grids result in the divergence of the discrete controls; see \cite{glowinski1995exact,infante1999boundary} and the survey paper \cite{zuazua2005propagation}. In this context, the authors in \cite{ervedoza2016numerical} propose restoring the uniform observability property by utilizing non-uniform grids. The main idea involves constructing various non-uniform finite-difference and finite-element semi-discretizations of the wave equation through the use of an appropriate concave diffeomorphism,  referred to as a reparametrization. It is important to mention that in \cite{ervedoza2016numerical}, the authors avoided addressing the uniform gap property related to the eigenfrequency of the discretized Laplace operator, even though, in general, this condition simplifies the study. This omission was due to the complexity of deriving information about the eigenfrequencies (or eigenvalues); thus, this property was identified only through numerical analysis. The first theoretical investigation of the uniform gap condition in this particular direction is attributed to the work \cite{bianchi2018spectral}, which demonstrated that it is indeed possible to derive information about the asymptotic distribution of the eigenfrequencies. Moreover, it was proved that using concave or convex reparametrizations of the domain results in a specific behavior of these eigenfrequencies, named the average gap property, which is a necessary condition for the uniform gap property, as established in \cite{bianchi2018spectral}. Furthermore, the study in \cite{bianchi2018spectral} is not confined to classical discretization, as in \cite{ervedoza2016numerical}, but deals with various schemes, including IGA. However, it should be mentioned that this average gap condition is insufficient to establish the wave equation's uniform boundary observability.

The main issue in both \cite{ervedoza2016numerical} and \cite{bianchi2018spectral} is the lack of information on the eigenfrequency spectrum of the discretized Laplace operator when using reparametrization. In this work, we address the following question: \textsl{How much information can be derived about the eigenfrequency spectrum when using a strictly concave or strictly convex reparametrization of the domain?} As in \cite{bianchi2018spectral}, we use Generalized Locally Toeplitz (GLT) theory to analyze the impact of these reparametrizations on the eigenfrequency distributions. The GLT theory was initially introduced in \cite{tilli1998locally} and further developed in \cite{capizzano2003generalized,serra2006glt}. The main advantage of GLT theory lies in its capacity to provide a general conceptual background for addressing problems within a broader context. Additionally, it can help determine whether a discretization method effectively approximates the spectrum; see \cite{garoni2019symbol,bianchi2021analysis}. Equipped with this theory and building on the results in \cite{bianchi2021analysis}, we illustrate the potential to gain valuable and previously undiscovered insights into the relationship between these specific reparametrizations and eigenfrequency behavior. Furthermore, we obtain more favorable outcomes than those in \cite{bianchi2021analysis}, including a uniform convergence version of \cite[Theorem 3.1]{bianchi2021analysis} and various new estimations of the eigenfrequencies.

\subsection{Main results}
Let \( p,n\in\mathbb{N}^* \), and \( N=n+p-2 \). We define the following admissible set of reparametrizations:  
\[
\mathbf{C}_{[0,1]}=\left\{ \phi\in\mathcal{C}^2([0,1]):\; \phi^{'}>0,\;[\phi^{''}>0 \text{ or }\phi^{''}<0],\; \phi(0)=0 \text{ and } \phi(1)=1       \right\}.
\]
We denote by \( L_{\phi,n}^p \) the IGA matrix of size \( N \) that approximates the one-dimensional Laplace eigenvalue problem under homogeneous Dirichlet boundary conditions \eqref{pd1}. Let \( \phi \in  \mathbf{C}_{[0,1]} \), we introduce the following functions:  
\[
    \Psi_{\phi}^{p}(y)=\mu_{2} \left( \left\{  (x,\theta)\in [0,1]\times[0,\pi]:\;\; \sqrt{\omega_{\phi}^p(x,\theta)}\leq y    \right\} \right),\;\;\forall y\in Rg\left({\sqrt{\omega_{\phi}^p}}\right),
\]
 and (see \cite{bianchi2018spectral} or Proposition \ref{stefa} below)  
\[
  \sqrt{\xi_{\phi}^p}(x)=\left(\Psi_{\phi}^{p}\right)^{-1}(\pi x),\quad \forall x\in[0,1],  
\]
where $\sqrt{\xi^p_{\phi}}$ is the GLT spectral symbol of $\sqrt{n^{-2}L_{\phi,n}^p}$ according to Definition \ref{defspectralsymbol}, see \cite{bianchi2018spectral}. Namely, we have 
\[
\sqrt{n^{-2}L_{\phi,n}^p}\sim_{\lambda}\sqrt{\xi^p_{\phi}}.
\]
Here, \( \mu_{2} \) denotes the Lebesgue measure on \( \mathbb{R}^{2} \), \( Rg \) represents the range of a function, and the function \( \omega^p_{\phi} \) is given by \eqref{eq:spectralsymbo-full-matrix}. Then, our main results can be summarized as follows:  

\begin{enumerate}
\item  We have established that the spectral symbol of the eigenfrequency exhibits a linear behavior near zero. For all \( \phi \) in \( \mathbf{C}_{[0,1]} \), Theorem \ref{symbolnear0} reveals that  
   \[
   \sqrt{\xi_{\phi}^p}(x)\sim_{0^+} \gamma x,
   \]
   where the constant \( \gamma \) is given by \eqref{gammadif}.

\item Based on \cite{bianchi2021analysis}, we derive Lemma \ref{approximationwithoutout}, which allows for the approximation of the eigenfrequency spectrum without outliers. We then use this lemma and the regularity of \( \Psi^p_{\phi} \) to establish an ordering relation, as illustrated in Theorem \ref{tdis1}. Specifically, for two distinct reparametrizations \( \phi_1 \) and \( \phi_2 \) in the admissible set \( \mathbf{C}_{[0,1]} \), and for any closed interval \( I \) in  
   \[
   Rg\left(\sqrt{\omega_{\phi_1}^p}\right) \cap Rg\left(\sqrt{\omega_{\phi_2}^p}\right),
   \]
   such that  
   \[
   \Psi_{\phi_1}^{p}(y)>\Psi_{\phi_2}^{p}(y),\quad \forall y\in  I,
   \]
   we can order the families of eigenfrequencies within \( I \). More precisely, let \( n \) be large enough, then for any \( k\in\{1,2,\dots,N\} \), such that  
   \[
   \sqrt{n^{-2}\lambda^{\phi_1}_{k,h}},\sqrt{n^{-2}\lambda^{\phi_2}_{k,h}}\in I,
   \]
   we have  
   \[
   \sqrt{\lambda_{k,h}^{\phi_1}}<\sqrt{\lambda_{k,h}^{\phi_2}}.
   \]

\item  We illustrate how the convexity of \( \Psi_{\phi}^p \) impacts the distribution of the eigenfrequency at a macro level. Theorem \ref{cl2.5} demonstrates that if \( \Psi_{\phi}^p \) is strictly convex on an interval, then all eigenfrequencies within this interval are shifted to its right. Conversely, if \( \Psi_{\phi}^p \) is strictly concave, then the eigenfrequencies concentrate on the left side of the interval.

\item We proved that the simple convergence in discrete Weyl's law, as presented in \cite[Theorem 3.1]{bianchi2021analysis}, is in fact uniform for all reparametrizations \( \phi \) in \( \mathbf{C}_{[0,1]} \). Specifically, Theorem \ref{gedis} indicates that the sequence of functions  
   \[
   G_n^p(y)= \dfrac{\left| \left\{k=1,\dots, N:\quad \sqrt{n^{-2}\lambda_{k,h}^{\phi}}\leq y\right\}\right|}{N+1},\quad \forall y\in Rg\left(\sqrt{\omega_{\phi}^p}\right),
   \]
   converges uniformly to  \( \frac{1}{\pi}\Psi^p_{\phi} \).
   
\item We leverage the asymptotic linearity of the symbol \( \sqrt{\xi^p_{\phi}} \) near zero to derive several eigenfrequency estimates, namely,  statement \( (ii) \) in Corollary \ref{equivalenceimportant}, \eqref{improved1} in Corollary \ref{improvedfrequincyestimation}, and \eqref{improved2} in Corollary \ref{improvedf}.
\end{enumerate}

In the study of the link between the reparametrization set \( \mathbf{C}_{[0,1]} \) and the distribution of eigenfrequency families, we constructed an infinite subset of \( \mathbf{C}_{[0,1]} \) formed by the reparametrizations defined in \eqref{eq:reparametrization-convex-example} and \eqref{eq:reparametrization-concave-example},  which enable the application of Theorems \ref{dis1} and \ref{dis2} to establish the ordering condition \eqref{hypoonordeing} in Theorem \ref{tdis1},  thereby ordering the families of eigenfrequencies. Furthermore, in the  special case where \( p=1 \), Proposition \ref{propc}, together with Theorem \ref{cl2.5}, indicates that the eigenfrequencies are shifted to the left within the interval  
\[
\left(0,\frac{\sqrt{6}}{\max\{\phi'(0),\phi'(1)\}}\right)
\]
for all \( \phi \) in \( \mathbf{C}_{[0,1]} \).

\subsection{Outline of the Paper}   

In Section \ref{sec:preliminaries}, we introduce the necessary notations, definitions, and preliminary results relevant to our analysis. Specifically, we start by defining  B-spline basis functions and briefly deriving the Galerkin Isogeometric Analysis (IGA) discretization for the one-dimensional Laplace eigenvalue  
problem. In  Subsection \ref{sec:preliminariesGLT} of this section, we provide an overview of the essential results from the abstract Generalized Locally Toeplitz (GLT) theory. At the end of this section, we present the GLT symbol for our specific discretization,  followed by the spectral symbol of the discretized Laplace eigenfrequency.  

In Section \ref{sec3:symbolanalysis}, we analyze the eigenfrequency symbol in a neighborhood of zero. We demonstrate that when applying a strictly convex or strictly concave reparametrization to the domain \( [0,1] \), the resulting symbol is \( C^2 \)-regular and strictly convex near zero. Additionally, it exhibits  linear behavior in this region.  

In Section \ref{sec4:impactofmapping}, we use GLT theory to examine the impact of reparametrization and symbol convexity on the distribution of eigenfrequencies. First, we establish Lemma \ref{approximationwithoutout}, which demonstrates the feasibility of approximating non-outlier eigenfrequencies by uniformly sampling the symbol. We then leverage this lemma, along with the regularity of the symbol’s inverse (see \cite{bianchi2018spectral} or Proposition \ref{stefa} below), to establish an ordering relation between families of eigenfrequencies generated by different reparametrizations. We conclude the section by establishing a shifting property for eigenfrequencies, which arises from the symbol's convexity.  

In Section \ref{improvedestimation}, we present an improvement over certain results in \cite{bianchi2021analysis}, specifically demonstrating a uniform version of discrete Weyl’s law. This result shows how simple convergence becomes uniform under a carefully chosen reparametrization. Additionally, we leverage the linearity of the symbol near zero to derive new, more accurate estimates of the eigenfrequencies via symbol sampling.  

Finally, Section \ref{sec:conclusions} summarizes the main findings of the paper and discusses possible future directions.

\section{Preliminaries}\label{sec:preliminaries}
Consider the following one-dimensional Laplace eigenvalue problem with homogeneous Dirichlet boundary conditions:
\begin{equation}\label{pd1}
\left\{
\begin{array}{ll}
 -\partial_{xx}u=\lambda u,\quad x\in (0,1),\\
 u(0)=u(1)=0.
\end{array}
\right.
\end{equation}
It is well-known that the system (\ref{pd1})  has a set of exact, non-trivial solutions given by $\lambda_k = \left( k\pi\right)^2$, $u_k(x) = \sin(k\pi x)$, for $k \in \mathbb{N}^*$. Here, $\lambda_k$ is the $k$-th eigenvalue of the operator $-\partial_{xx}$ with Dirichlet boundary conditions, and $u_k$ is the corresponding eigenfunction.  

This paper investigates the behavior of the IGA approximation of eigenfrequencies $(\sqrt{\lambda_k})_{k\geq1}$ when using strictly convex or concave reparametrizations to create a non-uniform spatial mesh. Specifically, we aim to understand how these mappings affect the distribution of the approximate eigenfrequencies and their estimation using uniform sampling of the IGA spectral symbol. The necessary background for our analysis is provided in this section, starting with the IGA Galerkin discretization of \eqref{pd1} in Subsection \ref{sec:descritization}, followed by a brief overview of GLT theory in Subsection \ref{sec:preliminariesGLT}.  Finally, we present the GLT symbol of the IGA eigenfrequencies in Subsection \ref{symbol}.

\subsection{Galerkin B-spline IGA Discretization}\label{sec:descritization}
System \eqref{pd1}, employing $B$-spline functions \cite{de1978practical}, involves discretizing the weak form of problem \eqref{pd1} stated as follows: for $k \geq 1$, find $u_k \in H^1_0(0,1)$ and $\lambda_k \in \mathbb{R}_+$ such that
\begin{equation}\label{pd2}
A(u_k,v)=\lambda_k \;L(u_k,v),\quad \forall v \in H^1_0(0,1),
\end{equation}
where 
$$ A(u_k,v)=\displaystyle\int_0^1u_k'(x)\;v'(x) \, dx,\;\;\text{and}\;\; L(u_k,v)=\displaystyle\int_0^1 u_k(x)\;v(x) \, dx. $$
The next step involves constructing a finite-dimensional subspace to approximate the solution space $H^1_0(0,1)$. This subspace is defined by a finite set of basis functions. In the standard IGA discretization, these functions are constructed using $B$-spline functions.

We consider non-periodic and uniform knot vectors of the form
$$
(t_j)_{0 \leq j \leq 2p+n}=\left( \displaystyle\underbrace{0\cdots0}_{p+1},\dfrac{1}{n},\dfrac{2}{n},\cdots,\dfrac{n-1}{n},\displaystyle\underbrace{1\cdots1}_{p+1} \right),
$$
where $n, p \in \mathbb{N}^*$. The $B$-spline functions of degree $p$ on these knots are defined recursively as follows (for instance, see \cite{de1978practical}): for $1 \leq k \leq p$,
\begin{equation}\label{eq:definition-B-spline-1}
N_j^k(t)=\dfrac{t-t_j}{t_{j+k}-t_{j}}N_j^{k-1}(t) +  \dfrac{t_{j+k+1}-t}{t_{j+k+1}-t_{j+1}}N_{j+1}^{k-1}(t),\quad\text{for}\;\;  0 \leq j \leq 2p+n-1-k,   
\end{equation}
with
\begin{equation}\label{eq:definition-B-spline-2}
N_j^0(t)=\mathcal{X}_{[t_j,t_{j+1})}(t),\quad\text{for}\;\;0 \leq j \leq 2p+n-1.
\end{equation}
Here, $p+n$ is the number of $B$-spline functions, and in cases where a fraction has a zero denominator, we assume it to be zero. We can then define the Schoenberg space
\begin{equation}\label{eq:Schoenberg-space}
\mathbb{S}^{p} = span \left\{ N_j^p: \; j=0, \dots,p+n-1 \right\}.
\end{equation}
In classical spline approximation theory (see \cite{de1978practical}), it is well-known that  $\mathbb{S}^{p}$ coincides with the space of splines of degree $p$ and smoothness $p-1$, namely
$$
\mathbb{S}^{p} = \left\{ s \in \mathcal{C}^{p-1}([0,1]),\,\,s|_{\left[ i/n, \, (i+1)/n \right)} \in \mathbb{P}_{p},\;\; i=0,\dots,n-1 \right\},
$$ 
where $\mathbb{P}_{p}$ denotes the space of polynomials of degree at most $p$.

The isogeometric approximation of $H^1_0([0,1])$ is given by 
\begin{equation}\label{eq:Schoenberg-Dirichlet-space}
\mathbb{S}^{p}_0 = \left\{ s \in \mathcal{C}^{p-1}([0,1]),\,\,s|_{\left[ i/n, \, (i+1)/n \right)} \in \mathbb{P}_{p},\;\;\text{and}\;\;s(0)=s(1)=0,\;\; i=0,\dots,n-1 \right\}.
\end{equation} 
Then, our discrete solutions  $(u_{k,h}, \lambda_{k,h}) \in \mathbb{S}^{p}_0 \times \mathbb{R}_+$ satisfy the approximate weak formulation
\begin{equation}\label{eq:discrete-weak-form}
A(u_{k,h},v_h)=\lambda_{k,h} \;L(u_{k,h},v_h),\quad \forall v_h \in \mathbb{S}^{p}_0,
\end{equation}
where $h$ refers to the discretization parameter defined by $h=1/n$. We use the standard basis for $\mathbb{S}^{p}_0$ formed by the $B$-spline functions $\{N_{1}^p, \cdots, N_{p+n-2}^p\}$ that vanish at the boundary. 
Equation \eqref{eq:discrete-weak-form} can be expressed as a finite-dimensional eigenvalue problem
\begin{equation*}
\left[\left( M_n^p\right)^{-1} K_n^p \right] \mathbf{u}_{k,h} = \lambda_{k,h} \mathbf{u}_{k,h},
\end{equation*}
where $\mathbf{u}_{k,h}$ is the coefficients vector of $u_{k,h}$ with respect to the  basis $\{N_{1}^p, \cdots, N_{p+n-2}^p\}$ and $M_n^p$ and $K_n^p$ are the mass and stiffness matrices, respectively 
\begin{equation*}
(M_n^p)_{i,j}= \int_0^1 N_i^p(x) N_j^p(x) \, dx,\quad(K_n^p)_{i,j} = \int_0^1 (N_i^p)'(x) (N_j^p)'(x) \, dx,
\end{equation*} 
for  $1 \leq i,j \leq p+n-2.$

As previously stated, this work aims to derive insights into the approximate eigenfrequencies when applying strictly convex or concave reparametrizations of the domain $[0,1]$. For the case without reparametrization, we refer the reader to \cite{ekstrom2018eigenvalues}. We then introduce the following test space:
\begin{equation}\label{eq:reparametrization-space}
\mathbf{C}_{[0,1]}=\left\{ \phi \in \mathcal{C}^2([0,1]):\; \phi^{'} > 0,\;[\phi^{''} > 0 \text{ or }\phi^{''} < 0],\; \phi(0) = 0 \text{ and } \phi(1) = 1 \right\}.
\end{equation}
For $\phi \in \mathbf{C}_{[0,1]} $, we define the following basis functions by pullback under transformation $\phi$:
$$
B_j^p = N_j^p \circ \phi^{-1}, \quad \text{for } 1 \leq j \leq p+n-2,
$$
and aim to approximate the exact eigenpairs $(u_k, \lambda_k)$ using the standard Galerkin method, where the discrete solution space is given by $span \{B_1^p, \cdots, B_{p+n-2}^p \}$. Simple computations lead to the following expressions for the mass and stiffness matrices:
\begin{equation}\label{eq:mass-matrix-dif}
(M_{\phi, n}^p)_{i,j}= \int_0^1 \left| \phi'(x) \right| N_i^p(x) N_j^p(x) \, dx, \quad \text{for } 1 \leq i,j \leq p+n-2.
\end{equation}
\begin{equation}\label{eq:stiffness-matrix-dif}
(K_{\phi,n}^p)_{i,j} = \int_0^1 \frac{1}{\left| \phi'(x) \right|} (N_i^p)'(x) (N_j^p)'(x) \, dx, \quad \text{for } 1 \leq i,j \leq p+n-2.
\end{equation}
The numerical eigenvalue problem in this case is described by:
\begin{equation*}
L_{\phi,n}^p \mathbf{u}_{k,h}^{\phi} = \lambda_{k,h}^{\phi} \mathbf{u}_{k,h}^{\phi},
\end{equation*}
where $L_{\phi,n}^p=\left( M_{\phi,n}^p\right)^{-1} K_{\phi,n}^p.$

\subsection{Preliminaries on GLT Sequences}\label{sec:preliminariesGLT}

This subsection provides a brief overview of the essential background on the Generalized Locally Toeplitz (GLT) sequences theory. More details can be found in the pioneering works \cite{bianchi2021analysis,garoni2017generalized}, and the references therein.  
In what follows, $(L_{n})_{n\in\mathbb{N}^{*}}$ represents a sequence of matrices of  size  $N=N(n)$, with $N \rightarrow +\infty$  as $n \rightarrow +\infty$. Furthermore, the eigenvalues of each matrix $L_n$, denoted by  $(\lambda_k(L_n))$, are assumed to be real, positive, and sorted in increasing order, specifically:  
$$0<\lambda_1(L_n)<\lambda_2(L_n)<\cdots<\lambda_N(L_n).$$  

In the current subsection, all the definitions and results have been adjusted to better align with our particular context, including the following definition of the spectral symbol.  
\begin{definition}[Spectral symbol]\label{defspectralsymbol}  
Let $\mathcal{C}_c(\mathbb{R})$ be the set of continuous functions with compact support over $\mathbb{R}$, and let $\omega :[0,1]\times[0,\pi] \longrightarrow \mathbb{R}$ be a measurable function. We say that $(L_{n})_{n\in\mathbb{N}^{*}}$ has a spectral (or eigenvalue) distribution described by $\omega$, and we write  
$$  
(L_{n})_{n\in\mathbb{N}^{*}}\sim_{\lambda} \omega,  
$$  
if for all $F \in \mathcal{C}_c(\mathbb{R})$ we have  
\begin{equation}\label{discretsspecralrelation}  
    \lim_{N\longrightarrow + \infty}\dfrac{1}{N}\displaystyle\sum_{k=1}^N F(\lambda_k(L_n))=\dfrac{1}{\pi}\displaystyle\int_{[0,1]\times[0,\pi]} F( \omega(x,\theta)) \,dx\,d\theta,  
\end{equation}  
where $\lambda_k(L_n)$, $k=1, \cdots, N$ are the eigenvalues of $L_n$. In this case, $\omega$ is referred to as the spectral symbol of $(L_{n})_{n\in\mathbb{N}^*}$.  
\end{definition}  
For insights into the largest set of test functions $F$ for which \eqref{discretsspecralrelation} holds, we refer to \cite{bianchi2024asymptotic}.  

With the chosen discretization in Subsection \ref{sec:descritization}, whether or not reparametrization is employed, an issue arises concerning outliers. A few eigenvalues are poorly approximated by the uniform sampling of the symbol $\omega$, and their corresponding values are notably larger than the exact values. The following definition  provides  the mathematical definition of these outliers.  
\begin{definition}[Outliers]  
Let $(L_{n})_{n\in\mathbb{N}^*}\sim_{\lambda} \omega$, and let $R_\omega$ represent the essential range of $\omega$, defined as   
\begin{equation}\label{eq:essential-range}  
R_{\omega} = \left\{ y \in \mathbb{R}:\; \mu_{2}\left(\{(x,\theta)\in [0,1]\times[0,\pi],\;\; |\omega(x,\theta)-y|<\epsilon      \}\right)> 0, \; \forall \epsilon>0 \right\},  
\end{equation}  
where $\mu_{2}$ denotes the Lebesgue measure on $\mathbb{R}^{2}$. An eigenvalue $\lambda_k(L_n)$ is considered an outlier if $\lambda_k(L_n) \notin  R_{\omega}$.  
\end{definition}  

It can be proved that $ R_{\omega}$ is a closed set (\cite[Lemma 2.1]{garoni2017generalized}). Furthermore, if the function $\omega$ is continuous  and  since the domain $[0,1]\times[0,\pi]$ is compact, we can demonstrate that the essential range of $\omega$ coincides with the image of $\omega$.  

In general, and specifically in our case, the symbol is defined on a multidimensional domain, which complicates the mathematical and numerical study of the distribution of the eigenvalues. However, we can  derive  a new one-dimensional symbol from the original symbol, as explained  in  the following  
\begin{definition}[Monotone rearrangement of the symbol]\label{def:monotone-rearrangement}  
Let $\omega: [0,1] \times [0,\pi] \longrightarrow \mathbb{R}$ be a measurable function such that $(L_n)_{n\in\mathbb{N}^*} \sim_{\lambda} \omega$. We assume that the essential range of $\omega$ is bounded. The extension function $\xi: [0,1] \longrightarrow R_{\omega}$  is  defined by  
$$  
\xi(x)=\inf\left\{ y\in  R_{\omega}:\;\; \Psi(y)>\pi  x  \right\},  
\quad \forall x\in(0,1),$$  
where   
\begin{equation}\label{eq:function-psi}  
\Psi(y)=\mu_{2}(\{  (x,\theta)\in [0,1]\times[0,\pi],\;\; \omega(x,\theta)\leq y    \}),\quad \forall y\in \mathbb{R},  
\end{equation}  
is called the monotone rearrangement of $\omega$.  
\end{definition}  

The following result \cite{bianchi2021analysis} shows that the monotone rearrangement of a symbol remains a spectral symbol for the same sequence of matrices.  

\begin{proposition}  
Let $(L_n)_{n\in\mathbb{N}^*} \sim_{\lambda} \omega$ with $\omega: [0,1] \times [0,\pi] \longrightarrow \mathbb{R}$  having a bounded essential range. Let $\xi$ be the monotone rearrangement of $\omega$. Then, we have  
$$   
(L_{n})_{n\in\mathbb{N}^{*}}\sim_{\lambda}\xi.  
$$  
\end{proposition}  

We conclude this subsection by introducing the discrete Weyl's law and some of its consequences. These results   describe the asymptotic behavior of eigenvalues  and are essential for our IGA spectral analysis  presented in the next section.  

\begin{theorem} [Discrete Weyl’s law, \cite{bianchi2021analysis}]\label{th1}  
Let $(L_n)_{n\in\mathbb{N}^*} \sim_{\lambda} \omega$ with $\omega: [0,1] \times [0,\pi] \longrightarrow \mathbb{R}$  having a bounded essential range. Define $\Psi:\, \mathbb{R} \longrightarrow \mathbb{R}_+$ as the function given by \eqref{eq:function-psi}. Then, at every point of continuity $y$ of $\Psi$, the eigenvalues of the matrices $L_n$ satisfy  
\begin{equation}\label{conversdiscretwels}  
    \lim_{n\rightarrow  +\infty} \dfrac{\left| \left\{k=1,\dots, N:\; \lambda_{k}(L_n)\leq y\right\} \right|}{N}=\frac{1}{\pi}\Psi(y),  
\end{equation}  

where, for a generic set $A$, $|A|$ denotes the number of elements in set $A$.  

Furthermore,  if we assume that $\Psi$ and $\xi$ are continuous, then for every sequence $k(n)\in\{1,\dots,N\}$ such that $\lim_{n\rightarrow  +\infty} \frac{k(n)}{N}=x\in[0,1]$ and $(\lambda_{k(n)}(L_n))_n\subset  R_\omega$, we have    
\begin{equation}\label{discretweylslaw}  
    \left(\frac{k(n)}{N}, \, \lambda_{k(n)}(L_n)  \right)\longrightarrow (x,\, \xi(x)), \; \text{ as } n \longrightarrow +\infty.  
\end{equation}  
\end{theorem}  

From the discrete Weyl's law,  we can deduce the following result, which demonstrates that the number of outliers is very small compared to $N$, specifically of order $o(N)$.  
\begin{corollary} \label{numberofoutlierscor}  
Under the hypotheses of Theorem \ref{th1}, if $\Psi$ is continuous, then:  
$$  
\lim_{n\rightarrow  +\infty}\dfrac{\left| \left\{k=1,\dots, N:\; \lambda_{k}(L_n)\notin R_{\omega} \right\}\right|}{N}=0.  
$$  
Moreover, for all $t\in R_{\omega}$ we have   
$$  
\lim_{n\rightarrow  +\infty}\dfrac{\left| \left\{k=1,\dots, N:\; \lambda_{k}(L_n)\leq t,\;\lambda_{k}(L_n)\in R_{\omega} \right\}\right|}{N}=\frac{1}{\pi}\Psi(t).  
$$  
\end{corollary}  

In the outlier-free context, we can approximate all the eigenvalues of matrices $L_n$ through the uniform sampling of the monotone rearrangement, as shown in the following result \cite{bianchi2021analysis}.  

\begin{corollary}\label{cl1}  
Under the hypotheses of Theorem \ref{th1}, and assuming additionally that $\Psi$ and $\xi$ are continuous, and that $R_{\omega}$ is bounded, then in the  absence of outliers, the error between the uniform sampling of $\xi$ and the eigenvalues of $L_{n}$ tends to $0$ as $n$  approaches infinity, namely  
$$  
\lim_{ n \rightarrow + \infty} \sup_{1 \leq k \leq N}\left\{ \left| \lambda_{k}(L_n) - \xi\left(\frac{k}{N+1}\right)\right|  \right\} = 0.  
$$  
\end{corollary}  

\begin{remark}  
In  Corollary \ref{cl1} above, the absence of outliers is an essential assumption. However, as illustrated in \cite[Chapter 5.1.2, p. 153]{cottrell2009isogeometric}, there is substantial numerical evidence pointing to the existence of outliers in Isogeometric Analysis (IGA) when B-splines of degree $p$ are employed. Specifically, it has been observed that the number of outliers depends only on $p$ and does not vary with the discretization step $h=1/n$. On the other hand, in \cite{manni2022application}, the authors proved that the optimal subspaces and the first reduced space are outlier-free. This illustrates that, in the absence of domain reparametrization, these spaces are the best choices for approximating the eigenvalues using the symbol (as shown in Corollary \ref{cl1}).  
\end{remark}  

In the subsequent sections of the paper, we will employ the notation $OUT(p, n)$ to represent the number of IGA outliers, where $OUT(p,n)=o(N)$, and  
\begin{equation}\label{setofnooutliers}  
    \mathcal{I}(p, n) = \left \{1, \cdots, N - OUT(p,n) \right \}  
\end{equation}  
 denotes the set of indices of eigenvalues after removing the outliers.

\subsection{The IGA GLT-Symbol}\label{symbol}
Before presenting the theorem that provides the spectral symbol of the matrix \(L^{p}_{\phi,n}\), it is essential to introduce what we call the  cardinal B-spline function \cite{de1978practical}. Let \(\mathcal{N}_p: \mathbb{R}\longrightarrow\mathbb{R} \) be the cardinal B-spline of degree \(p\), recursively defined as follows:  

\begin{equation}\label{eq:cardinal-B-spline-1}  
\mathcal{N}_0(x)=\left\{  
\begin{array}{ll}  
1,\;\;x\in[0,1],\\  
0,\;\;\text{otherwise},  
\end{array}  
\right.  
\vspace{-0.25cm}    
\end{equation}  

and  

\begin{equation}\label{eq:cardinal-B-spline-2}  
\mathcal{N}_p(x)=\frac{x}{p}\mathcal{N}_{p-1}(x)+\frac{p+1-x}{p}\mathcal{N}_{p-1}(x-1),\;\;x\in\mathbb{R},\;\;p\geq 1.  
\end{equation}  

It  has been shown in \cite{de1978practical} that \(\mathcal{N}_p\in C^{p-1}(\mathbb{R})\) and  
\[
supp(\mathcal{N}_p)=[0,p+1].
\]  
For \(p\in\mathbb{N}\),  we define the functions:  
\[
f_p:[0,\pi]\longrightarrow\mathbb{R}, \quad f_p(\theta)=-\mathcal{N}^{''}_{2p+1}(p+1)-2\displaystyle\sum_{k=1}^p\mathcal{N}^{''}_{2p+1}(p+1-k)\cos(k\theta),\;\;p\geq 1,
\]  
\[
g_p:[0,\pi]\longrightarrow\mathbb{R}, \quad g_p(\theta)=\mathcal{N}_{2p+1}(p+1)+2\displaystyle\sum_{k=1}^p\mathcal{N}_{2p+1}(p+1-k)\cos(k\theta),\;\;p\geq 0,
\]  
and  
\[
e_p:[0,\pi]\longrightarrow\mathbb{R},\quad e_p(\theta)=\dfrac{f_p(\theta)}{g_p(\theta)},\;\;p\geq 1.
\]  
It is known from \cite{garoni2014spectrum} that  
\[
f_p(\theta)=(2-2\cos(\theta))g_{p-2}(\theta),\;\;\theta\in[0,\pi],\;\; p\geq 2.
\]  

\[
\left( \frac{4}{\pi^2}\right)^{p+1}\leq g_p(\theta)\leq g_p(0)=1,\;\; \theta\in[0,\pi],\;\; p\geq 0.
\]  

We can  also see from \cite{donatelli2016spectral} (see also \cite{bianchi2018spectral}) that for every \(p\geq 2\), it holds that  

\begin{equation}\label{eq:preliminaries-1}  
e_p(\theta)=(2-2\cos(\theta))\frac{g_{p-2}(\theta)}{g_{p}(\theta)},\;\; \left( \frac{2}{\pi} \right)^{p-1}\leq\frac{g_{p-2}(\theta)}{g_{p}(\theta)}\leq \left( \frac{\pi}{2} \right)^{p+1},\;\; \theta\in[0,\pi].
\end{equation}  

$\hspace{0.52cm}$Using the inner-product property of cardinal B-splines, it has been demonstrated that both \(M_{\phi,n}^p\) and \(K_{\phi,n}^p\) are  small rank perturbations of Toeplitz matrices. A complete proof of this result can be found in \cite{garoni2017generalized}.  By applying the GLT theory, we obtain the following theorem:  

\begin{theorem}[\cite{garoni2017generalized}, IGA GLT symbol]  
Let \(p\geq 1\) and \(\phi\in \mathbf{C}_{[0,1]}\). Then  
\[
n^{-2}L^{p}_{\phi,n}\sim_{\lambda}\omega^p_{\phi},
\]  
where    
\begin{equation}\label{eq:spectralsymbo-full-matrix}  
\omega^p_{\phi}(x,\theta)=\dfrac{e_p(\theta)}{\left(\phi^{'}(x)\right)^2},\quad\forall (x,\theta)\in [0,1]\times[0,\pi].
\end{equation}  
\end{theorem}  

The following result illustrates the regularity and some properties of the symbol   $e_p$ (see \cite{ekstrom2018eigenvalues}).
\begin{corollary}\label{eppropreties}
Let $p\geq 1$. The function $e_p$ is differentiable, nonnegative, and monotonically increasing on the interval $[0,\pi]$. Additionally, it satisfies the following properties:
$$
 e_p(\theta) \sim_{\theta \to 0^{+}} \theta^2, \quad \text{as}\quad \lim_{p\rightarrow +\infty }\sup_{\theta\in[0,\pi]}|e_p(\theta)-\theta^2|=0.
$$
\end{corollary}

In what follows, our focus will be on the frequency analysis of the matrix $L^{p}_{\phi,n}$. To simplify the notation, we refer to the size of $L^{p}_{\phi,n}$ by $N$, such that $N=n+p-2$. Let us introduce the following functions:
For all $\phi\in \mathbf{C}_{[0,1]}$, 
\begin{equation}\label{largpsidif}
    \Psi_{\phi}^{p}(y)=\mu_{2} \left( \left\{  (x,\theta)\in [0,1]\times[0,\pi]:\;\; \sqrt{\omega_{\phi}^p(x,\theta)}\leq y    \right\} \right),\;\;\forall y\in  R_{\sqrt{\omega_{\phi}^p}},
\end{equation}
and 
\begin{equation}\label{symboldif}
  \sqrt{\xi_{\phi}^p}(x)=\inf\left\{ y\in   R_{\sqrt{\omega_{\phi}^p}}:\;\; \Psi_{\phi}^{p}(y)>\pi  x  \right\},\;\;\forall x\in(0,1).  
\end{equation}
Based on  \cite[Theorem 10.16]{garoni2017generalized}, we derive using Generalized Locally Toeplitz (GLT) theory that the frequency distribution of the matrices $(n^{-2}L_{\phi,n}^p)$ follows $\sqrt{\xi^p_{\phi}}$, a property that has already been used in \cite{bianchi2018spectral}. Specifically, we have 
\begin{equation}\label{spectarlfrequincy}
    \sqrt{n^{-2}L_{\phi,n}^p}\sim_{\lambda}\sqrt{\xi^p_{\phi}},
\end{equation}
where $\sqrt{\xi_{\phi}^p}$ is given by \eqref{symboldif}, the monotone rearrangement of $\sqrt{\omega_\phi^p}$. Note that in our case, $ \sqrt{\omega_{\phi}^p} $ is a continuous function, which implies that the essential range of $ \sqrt{\omega_{\phi}^p} $ coincides with its regular range. Furthermore, it is apparent in this scenario that $\Psi_{\phi}^{p}$ is strictly increasing over $Rg\left(\sqrt{\omega_{\phi}^p}\right)$.

We now conclude the preliminaries section with the following result, which plays a crucial role in our analysis. For the proof, see \cite{bianchi2018spectral}.

\begin{proposition}\label{stefa}
For all  $\phi\in \mathbf{C}_{[0,1]}$ and for all $p\geq 1$, we have 
$$
\Psi_{\phi}^{p}\in \mathcal{C}^1((0,+\infty))\;\;\text{and} \;\; \sup_{y \geq 0} \left(\Psi^{p}_{\phi}\right)'(y)<\infty.
$$ 
Moreover,
\begin{equation}\label{relatioxiandPsi}
   \sqrt{\xi_{\phi}^p}(x)=\left(\Psi_{\phi}^{p}\right)^{-1}(\pi x),\;\;\forall x \in [0,1],
\end{equation}
where $\Psi_{\phi}^{p}$ and $\sqrt{\xi_{\phi}^p}$ are given by \eqref{largpsidif} and \eqref{symboldif}, respectively.
\end{proposition}


\section{Symbol analysis}\label{sec3:symbolanalysis}
In this section, we examine the behavior of the symbol of the eigenfrequency near zero. We first establish that $\sqrt{\xi_{\phi}^1}$, corresponding to $p=1$, is $C^2$-regular and strictly convex in a neighborhood of zero. We then leverage this analysis to derive the asymptotic behavior of the symbol $\sqrt{\xi_{\phi}^p}$ near zero for all $p$. The symbol analysis carried out in this section is crucial for the subsequent sections. For the moment, the results regarding the symbol $\sqrt{\xi_{\phi}^1}$ serve as an example application of Theorem \ref{cl2.5} presented in the next Section \ref{sec4:impactofmapping}. Furthermore, the linearity of the symbol near zero is a key result for the eigenfrequency estimations established in Section \ref{improvedestimation}.

Throughout this section, we choose $\phi\in\mathbf{C}_{[0,1]}$ to be strictly convex. However, with slight modifications to the proofs, similar results can be established when the reparametrization is strictly concave.


\begin{proposition}\label{propc}
Let $\phi$ be a strictly convex reparametrization of the interval $[0,1]$. Then for every $\epsilon\in(0,1)$, the function $\Psi_{\phi}^{1}$ is $C^{2}\left(\left(0,\epsilon\frac{\sqrt{12}}{\phi^{'}(1)}  \right) \right)$ and \\ $\left(\Psi_{\phi}^{1}\right)^{'}(0)=1$. In addition, $\Psi_{\phi}^{1}$ is strictly concave over $\left(0,\frac{\sqrt{6}}{\phi^{'}(1)}\right)$.
\end{proposition}

\begin{proof}
Let $\phi\in\mathbf{C}_{[0,1]}$ to be strictly convex and $y\in \mathbb{R}_+$. As a first step, we obtain an explicit expression for $\Psi_{\phi}^{1}(y)$. We proceed as follows
\begin{align}\label{eq:propc-eq1}
\Psi_{\phi}^{1}(y)&= \mu_2\left\{ \sqrt{\omega_{\phi}^1}\leq y\right\} \nonumber\\
& =\mu_2\left\{ (x,\theta)\in [0,1]\times[0,\pi],\;\;\; \omega_{\phi}^1(x,\theta)\leq y^2 \right\} \nonumber \\
 & =\mu_2\left\{ (x,\theta)\in [0,1]\times[0,\pi],\;\;\; \dfrac{6(1-\cos{\theta})}{2+\cos{\theta}}\leq \left(y \phi^{'}(x)\right)^2 \right\} \nonumber\\
 &= \mu_2\left\{ (x,\theta)\in [0,1]\times[0,\pi],\;\;\; 6(1-\cos{\theta})\leq 2 \left(y \phi^{'}(x)\right)^2 + \left(y \phi^{'}(x)\right)^2 \cos{\theta} \right\}\nonumber\\
 &= \mu_2\left\{ (x,\theta)\in [0,1]\times[0,\pi],\;\;\; \cos{\theta}\geq \beta_x \right\}.
\end{align}
Here, $\beta_x$ is defined as
$$
\beta_x = \dfrac{6-2 \left(y\phi^{'}(x)\right)^2}{6+\left(y \phi^{'}(x)\right)^2}, \quad 0 \leq x \leq 1. 
$$
To evaluate the last measure \eqref{eq:propc-eq1}, we employ the property that $\cos:\, [0, \pi] \longrightarrow [-1,1]$ is invertible, $\arccos$ being its inverse function. This requires characterizing the conditions under which $\beta_x$ is in $[-1, 1]$. In fact, it is easy to see that $\beta_x \leq 1$ for all $x \in [0,1]$, and $\beta_x \geq -1$ if, and only if
\begin{equation*}
    \left\{
\begin{array}{ll}
0 \leq x \leq 1 \\
0 \leq y \leq \frac{\sqrt{12}}{\phi'(1)}, 
 \end{array}
\right. \text{or} \quad \left\{
\begin{array}{ll}
0 \leq x \leq \left(\phi' \right)^{-1} \left( \frac{\sqrt{12}}{y} \right)\\
\frac{\sqrt{12}}{\phi'(1)} \leq y \leq \frac{\sqrt{12}}{\phi'(0)}.
\end{array}
\right.
\end{equation*}
Therefore, from \eqref{eq:propc-eq1}, we deduce that
\begin{equation}\label{psiequa}
  \Psi_{\phi}^{1}(y)=\left\{
\begin{array}{ll}
  \displaystyle\int_{0}^{1}\arccos\left(  \dfrac{6-2\left( y\phi'(x)\right)^2}{6+\left( y\phi'(x)\right)^2}\right)\;dx, &y\in J_1,\\\\

\displaystyle\int_{0}^{\left( \phi'\right)^{-1}\left(\dfrac{\sqrt{12}}{y}\right)}\arccos\left(  \dfrac{6-2\left( y\phi'(x)\right)^2}{6+\left( y\phi'(x)\right)^2}\right)\;dx \\ \hspace{4cm}+\pi\left(1-\left( \phi'\right)^{-1}\left(\dfrac{\sqrt{12}}{y}\right)\right),\\ 
&y\in J_2,
 \end{array}
\right.
\end{equation}
where $J_1:= \left(0,\dfrac{\sqrt{12}}{\phi'(1)}\right)$ and $J_2:=  \left(\dfrac{\sqrt{12}}{\phi'(1)},\dfrac{\sqrt{12}}{\phi'(0)}\right)$. Since $\left(\Psi_{\phi}^{1}\right)'$ is $C^1((0,+\infty[)$ by Proposition \ref{stefa}, we then compute $\left(\Psi_{\phi}^{1}\right)'$ on $J_1$. 

Let $y \in J_1$, we define
$$  
\varphi(x,y)= \arccos\left(  \dfrac{6-2\left( y\phi'(x)\right)^2}{6+\left( y\phi'(x)\right)^2}\right),\;\; \forall x \in [0,1],\quad y \in J_1.
$$
For almost every $x \in [0,1]$ and all $y \in J_1$, we have
\begin{align*}
\dfrac{\partial \varphi}{\partial y}(x,y)&=\dfrac{-1}{\sqrt{1-\left[ \dfrac{6-2\left( y\phi'(x)\right)^2}{6+\left( y\phi'(x)\right)^2}\right]^2}}\;\dfrac{-36 y \left(\phi'(x)\right)^2}{\left[ 6+\left( y\phi'(x)\right)^2\right]^2}\vspace*{0.25cm}\\
&=\dfrac{36 y \left(\phi'(x)\right)^2}{\sqrt{ \left(12-\left( y\phi'(x)\right)^2\right) \left[ 3\left( y\phi'(x)\right)^2\right]    }}\;\dfrac{1}{ 6+\left( y\phi'(x)\right)^2}\vspace*{0.25cm}\\
&=\dfrac{\phi'(x)}{\sqrt{1-\left(\dfrac{y\phi'(x)}{\sqrt{12}}\right)^2}}\; \dfrac{6}{6+\left( y\phi'(x)\right)^2} \vspace*{0.25cm}\\
& \leq \dfrac{\phi'(x)}{\sqrt{1-\left(\dfrac{\phi'(x)}{\phi'(1)}\right)^2}}  = \dfrac{\phi'(x) \sqrt{\phi'(1)}}{\sqrt{1+\dfrac{\phi'(x)}{\phi'(1)}}} \;\dfrac{1}{\sqrt{\phi'(1)-\phi'(x)}} \vspace*{0.25cm}\\
& \leq \frac{\left(\phi'(1)\right)^{3/2}}{\sqrt{1+\frac{\phi'(0)}{\phi'(1)}}} \;\dfrac{1}{\sqrt{\phi'(1)-\phi'(x)}},
\end{align*}
and
$$
\int_0^1 \dfrac{dx}{\sqrt{\phi'(1)-\phi'(x)}}    \leq \frac{1}{\displaystyle\inf_{[0,1]} \left\{\phi''\right\}} \int_{\phi'(0)}^{\phi'(1)} \dfrac{dx}{\sqrt{\phi'(1)-x}}  < \infty.
$$
Thus, by the Lebesgue Dominated Convergence Theorem, $\Psi_{\phi}^{1} \in C^1(J_1)$, and
\begin{equation}\label{PsiJ1}
\left(\Psi_{\phi}^{1}\right)'(y)=\displaystyle\int_0^1 \dfrac{\partial \varphi}{\partial y}(x,y)\;dx,\quad \forall y\in J_1,\quad \left(\Psi_{\phi}^{1}\right)'(0)=1.
\end{equation}

For the third step, we fix $\epsilon$  in $(0,1)$ and we aim to demonstrate  that $\left(\Psi_{\phi}^{1}\right)'$ is of class $C^{1}\left(J_1^{\epsilon}\right)$, where $$J_1^{\epsilon}=\left(0,\epsilon\frac{\sqrt{12}}{\phi^{'}(1)}  \right).$$

Let $y\in J_1^{\epsilon}$. From (\ref{PsiJ1}), we observe that for every $x\in [0,1]$ the function $\frac{\partial\varphi}{\partial y}$ is differentiable with respect to $y$ and we have 
\begin{align*}
    \frac{\partial^2\varphi}{\partial^2 y}(x,y)&=-6\phi^{'}(x)\dfrac{\frac{-\left(\frac{\phi^{'}(x)}{\sqrt{12}}\right)^2 2y}{2\sqrt{1-\left(\frac{y\phi^{'}(x)}{\sqrt{12}}\right)^2}}\left(6+\left(y\phi^{'}(x)\right)^2\right)+\sqrt{1-\left(\frac{y\phi^{'}(x)}{\sqrt{12}}\right)^2}\;2y\left(\phi^{'}(x)\right)^2}{\left(1-\left(\frac{y\phi^{'}(x)}{\sqrt{12}}\right)^2\right)\left(6+\left(y\phi^{'}(x)\right)^2\right)^2}\\
    &=-6y\left(\phi^{'}(x)\right)^3\dfrac{-\frac{1}{12}\left(6+\left(y\phi^{'}(x)\right)^2\right)+2\left(1-\left(\frac{y\phi^{'}(x)}{\sqrt{12}}\right)^2\right)}{\left(1-\left(\frac{y\phi^{'}(x)}{\sqrt{12}}\right)^2\right)^{\frac{3}{2}}\left(6+\left(y\phi^{'}(x)\right)^2\right)^2}\\
    &=\dfrac{-9y\left(\phi^{'}(x)\right)^3 \left(1-2\left(\frac{y\phi^{'}(x)}{\sqrt{12}}\right)^2\right)}{\left(1-\left(\frac{y\phi^{'}(x)}{\sqrt{12}}\right)^2\right)^{\frac{3}{2}}\left(6+\left(y\phi^{'}(x)\right)^2\right)^2}.
\end{align*}
Observe that $\frac{\partial^2\varphi}{\partial^2 y}$ is continuous over $[0,1]\times J_1^{\epsilon} $, then  $\Psi_{\phi}^{1}\in C^2(J_1^{\epsilon})$, with
\begin{equation}\label{PsipJ1}
\left(\Psi_{\phi}^{1}\right)^{''}(y)=\displaystyle\int_0^1 \dfrac{\partial^2 \varphi}{\partial^2 y}(x,y)\;dx,\quad\forall y\in J_1^{\epsilon}.
\end{equation}

To conclude the proof, let $y\in \left(0,\frac{\sqrt{6}}{\phi^{'}(1)}\right)$. Then, there exists an $\epsilon$ in $(0,1)$ such that 
\begin{align*}
    y&\leq \frac{\sqrt{6}}{\phi^{'}(1)}< \epsilon\frac{\sqrt{12}}{\phi^{'}(1)}\\
    &\Rightarrow y\in J_1^{\epsilon}\;\;\text{ and}\;\; 1-2\left(\frac{y\phi^{'}(x)}{\sqrt{12}}\right)^2\geq 1-\left(\frac{\phi^{'}(x)}{\phi^{'}(1)}\right)^2\geq 0.
\end{align*}
Hence, using the (\ref{PsipJ1}), we obtain 
$$\left(\Psi_{\phi}^{1}\right)^{''}(y)<0,\;\;\forall y\in \left(0,\frac{\sqrt{6}}{\phi^{'}(1)}\right). $$

Finally, the function $\Psi_{\phi}^{1}$ is strictly concave over $\left(0,\frac{\sqrt{6}}{\phi^{'}(1)}\right)$, which ends the proof.
\end{proof}


Using relation \eqref{relatioxiandPsi}, which links $\sqrt{\xi^1_{\phi}}$ with $\Psi^1_{\phi}$, we obtain the following results that characterize the eigenfrequency symbol in the neighborhood of zero.

\begin{corollary}
    Let $\phi$ be a strictly convex reparametrization of the interval $[0,1]$. Then for every $\epsilon\in(0,1)$, the symbol $\sqrt{\xi_{\phi}^1}$ is $C^2$ on $\left(0,\frac{1}{\pi}\Psi_{\phi}^1\left(\epsilon\frac{\sqrt{12}}{\phi^{'}(1)}  \right) \right)$ and strictly convex over $\left(0,\frac{1}{\pi}\Psi_{\phi}^1\left(\frac{\sqrt{6}}{\phi^{'}(1)}\right)\right)$.
\end{corollary}

\begin{proof}
    The result follows directly from the above Proposition \ref{propc}. We combine (\ref{relatioxiandPsi}) with the fact that $\left(\Psi_{\phi}^{1}\right)'(y)>0$ over $\left(0,\epsilon\frac{\sqrt{12}}{\phi^{'}(1)}  \right)$ (see (\ref{PsiJ1})). For brevity, the details are omitted.
\end{proof}

We are now in a position to state the main result of this section. Using Proposition \ref{stefa} and relation \eqref{relatioxiandPsi}, the main idea of the theorem below centers on establishing $\left(\Psi_{\phi}^{p}\right)'(0)\neq 0$ for all $p\geq 1$. In this regard, it is important to note the significance of the above analysis for the case $p=1$. First, from Proposition \ref{propc}, we have $\left(\Psi_{\phi}^{1}\right)'(0)=1$; second, the techniques used to derive the general result for $p\geq1$ are feasible only when $p\geq2$. 

\begin{theorem}\label{symbolnear0}
    Consider $p\in \mathbb{N}^{*}$ and $\phi\in \mathbf{C}_{[0,1]}$ to be a convex reparametrization  of $[0,1]$. Then, the symbol $\sqrt{\xi_{\phi}^p}$ is linear in the neighborhood of zero. Precisely, we have 
    \begin{equation}\label{assymptoticatzero}
        \sqrt{\xi_{\phi}^p}(x)\sim_{0^+} \gamma x,
    \end{equation}
where, 
\begin{equation}\label{gammadif}
    \gamma:= \frac{\pi}{\left(\Psi_{\phi}^{p}\right)'(0)}>0.
\end{equation}
\end{theorem}

\begin{proof} 

 We will divide this proof into two cases. The first case is when $p=1$. From \eqref{relatioxiandPsi}, we have 
  \begin{equation*}
\sqrt{\xi_{\phi}^1}(x)=\left(\Psi_{\phi}^{1}\right)^{-1}(\pi x),\quad \forall x\in[0,1].
\end{equation*}
Then, using Proposition \ref{stefa} and  the fact that $\left(\Psi_{\phi}^{1}\right)'(0)=1\neq 0$ ( see Proposition \ref{propc}) and $\Psi_{\phi}^{1}(0)=0$, the symbol  $\sqrt{\xi_{\phi}^1}$ is differentiable at $0$, and we have
$$
\left(\sqrt{\xi_{\phi}^1}\right)^{'}(0)=\pi \dfrac{1}{\left(\Psi_{\phi}^{1}\right)^{'}\left(\left(\Psi_{\phi}^{1}\right)^{-1}(0)\right)}.
$$
This yields (\ref{assymptoticatzero}) and ends the proof for $p=1.$
   
Now we focus on the case $p\geq 2$. Using relations \eqref{eq:preliminaries-1} from Subsection \ref{symbol}, we derive the following inequalities for all $\theta\in[0,\pi]$
$$
2\sin\left(\frac{\theta}{2}\right)\left(\frac{2}{\pi}\right)^{(p-1)/2}\leq\sqrt{e_p(\theta)}\leq 2\sin\left(\frac{\theta}{2}\right)\left(\frac{\pi}{2}\right)^{(p+1)/2},\quad \forall \theta\in[0,\pi].
$$
Given that for all $y\in Rg\left(\sqrt{\omega^p_{\phi}}\right) $, we have
$$
\Psi_{\phi}^{p}(y)=\mu_{2} \left( \left\{  (x,\theta)\in [0,1]\times[0,\pi]:\quad \frac{\sqrt{e_p(\theta})}{\phi^{'}(x)}\leq y    \right\} \right),
$$
we deduce 
\begin{equation}\label{eq:linb-1}
\Psi_{1}(y)\leq \Psi_{\phi}^{p}(y)\leq \Psi_{2}(y),
\end{equation}
where
$$
\Psi_{1}(y)=\mu_{2} \left( \left\{  (x,\theta)\in [0,1]\times[0,\pi]:\quad \frac{2\sin\left(\frac{\theta}{2}\right)\left(\frac{\pi}{2}\right)^{(p+1)/2}}{\phi^{'}(x)}\leq y    \right\} \right),
$$
and 
$$
\Psi_{2}(y)=\mu_{2} \left( \left\{  (x,\theta)\in [0,1]\times[0,\pi]:\quad \frac{2\sin\left(\frac{\theta}{2}\right)\left(\frac{2}{\pi}\right)^{(p-1)/2}}{\phi^{'}(x)}\leq y    \right\} \right).
$$
According to \cite[Proposition 1]{bianchi2018spectral}, $\Psi_1$ and $\Psi_2$ are right-differentiables at $0$, and using relation (37) in the proof \cite[Proposition 1]{bianchi2018spectral}, we obtain 
$$
\Psi_1'(0) = \left( \frac{2}{\pi} \right)^{(p+1)/2}, \quad \text{and} \quad \Psi_2'(0) = \left( \frac{\pi}{2} \right)^{(p-1)/2}.
$$
Hence, by dividing both sides of \eqref{eq:linb-1} by $y$ and taking the limit as $y \rightarrow 0^+$, we get
$$
0<\left(\frac{2}{\pi}\right)^{(p+1)/2}\leq \left(\Psi_{\phi}^{p}\right)'(0)\leq\left(\frac{\pi}{2}\right)^{(p-1)/2}.
$$
Using \eqref{relatioxiandPsi} and $\Psi_{\phi}^{p}(0)=0$ ( see \eqref{largpsidif}), we obtain 
$$
\left(\sqrt{\xi_{\phi}^p}\right)'(0)=\pi \dfrac{1}{\left(\Psi_{\phi}^{p}\right)'\left(\left(\Psi_{\phi}^{p}\right)^{-1}(0)\right)},
$$ 
which concludes the proof of the theorem.
\end{proof}
\section{Impact of reparametrization on the behavior of eigenfrequencies}\label{sec4:impactofmapping}
When choosing a reparametrization $\phi_1$ in $\mathbf{C}_{[0,1]}$, we generate through IGA discretization (see Subsection \ref{sec:descritization}) a family of eigenfrequencies, denoted as $\left(\sqrt{\lambda_{k,h}^{\phi_1}}\right)_{k \in \mathcal{I}(p,n)}$, which depends, by construction, on the mapping $\phi_1$. Similarly, when using another reparametrization $\phi_2\in\mathbf{C}_{[0,1]}$, we generate a new family of eigenfrequencies, denoted as $\left(\sqrt{\lambda_{k,h}^{\phi_2}}\right)_{k \in \mathcal{I}(p,n)}$. In this section, we aim to analyze the impact of changing the mapping on the distribution of eigenfrequencies. Specifically, we aim to locate each family with respect to the other. Additionally, we want to understand if choosing a reparametrization that results in a particular property of the symbol influences the behavior of the eigenfrequencies. We begin our analysis with a lemma that allows us to approximate the non-outlier eigenfrequencies by the symbol sampling. Following this, we establish an ordering relation between the eigenfrequencies and the reparametrizations set $\mathbf{C}_{[0,1]}$. Finally, we analyze how the symbol's convexity impacts the eigenfrequency behavior.

In this part, we chose to present our findings using conditions on $\Psi^p_{\phi}$ due to its simpler definition compared to $\sqrt{\xi_{\phi}^p}$. However, it is important to note that the two are linked through \eqref{relatioxiandPsi}.

\begin{lemma}\label{approximationwithoutout}
Let $\phi\in\mathbf{C}_{[0,1]}$, such that  $\sqrt{n^{-2}L_{\phi,n}^p}\sim_{\lambda} \sqrt{ \xi^p_{\phi}}$. Then
\begin{equation}\label{eq:lemma-ditribution-eig-out}
\lim_{n \rightarrow +\infty}\;\max_{k \in \mathcal{I}(p,n)}\left|    \sqrt{n^{-2}\lambda_{k,h}^{\phi}}   -  \left(\Psi^{p}_{\phi}\right)^{-1}\left(\frac{k\pi}{N+1}\right)\right|= 0.
\end{equation}
\end{lemma}

\begin{proof}
    The proof follows a similar approach to that used in the proof of Corollary 3.3 in \cite{bianchi2021analysis}. We assume that (\ref{eq:lemma-ditribution-eig-out}) is false. Under this assumption, we observe that all subsequences $(k(n))_{n\geq1}$ are contained within $\mathcal{I}(p,n)$, which implies, by construction, that $\left( \sqrt{n^{-2}\lambda_{k(n),h}^{\phi}}\right)_{n\geq1}$ is contained within $Rg\left(\sqrt{\omega_{\phi}^p}\right)$. Applying (\ref{discretweylslaw}) directly leads to a contradiction, which completes the proof.
\end{proof}

We now present the first main result of this section.
\begin{theorem}[Distribution of eigenvalues]\label{tdis1}
Let $\phi_1, \phi_2 \in \mathbf{C}_{[0,1]}$, and let $I$ be a closed interval in  $Rg(\omega_{\phi_1}^p) \cap Rg(\omega_{\phi_2}^p)$. If, for all $y \in I$, it holds
\begin{equation}\label{hypoonordeing}
  \Psi_{\phi_1}^{p}(\sqrt{y})>\Psi_{\phi_2}^{p}(\sqrt{y}). 
\end{equation}

Then, there exists $  n_0\in\mathbb{N}^{*}$,  such that for all $n\geq n_0$  and $k\in \mathbb{N}^{*}$, if \\$n^{-2}\lambda_{k,h}^{\phi_1},n^{-2}\lambda_{k,h}^{\phi_2}\in I$, we have 
$$ 
\sqrt{\lambda_{k,h}^{\phi_1}}<\sqrt{\lambda_{k,h}^{\phi_2}},
$$ 
where, for $i=1,2$, $\left(\sqrt{\lambda_{k,h}^{\phi_i}}\right)_k$ are the eigenfrequencies generated when using $\phi_i$ as a reparametrization.
\end{theorem}

\begin{proof}
We will begin by proving the following lemma:
\begin{lemma}\label{theo:lemma-ditribution-eig-1}
For all $n \in \mathbb{N}^*$ and $k \in \mathcal{I}^{1,2}(p,n)=\mathcal{I}^{\phi_1}(n,p)\cap\mathcal{I}^{\phi_2}(n,p)$, there exist unique $z_{k,h}^{\phi^1}$ and $z_{k,h}^{\phi^2}$ in the interval $[0, \pi]$ such that:
$$\sqrt{n^{-2}\lambda_{k,h}^{\phi_1}}= \left(\Psi^{p}_{\phi_1}\right)^{-1}(z_{k,h}^{\phi_1}),\; \sqrt{n^{-2}\lambda_{k,h}^{\phi_2}}= \left(\Psi^{p}_{\phi_2}\right)^{-1}(z_{k,h}^{\phi_2})$$
and 
$$
\lim_{n\rightarrow +\infty}\; \max_{k \in \mathcal{I}^{1,2}(n,p)} \left| z_{k,h}^{\phi_1}-   z_{k,h}^{\phi_2}  \right|=0.
$$
\end{lemma}

\begin{proof} (of Lemma \ref{theo:lemma-ditribution-eig-1}).
Let $n \in \mathbb{N}^*$ and $k \in \mathcal{I}^{1,2}(p,n)$. Then, $n^{-2}\lambda_{k,h}^{\phi_1} \in Rg (\omega_{\phi_1}^p )$, $n^{-2}\lambda_{k,h}^{\phi_2} \in Rg (\omega_{\phi_2}^p )$. Since, for $i=1,2$, the function $\left(\Psi^{p}_{\phi_i}\right)^{-1}$ is continuous and strictly increasing , there exist unique $z_{k,h}^{\phi^1}$ and $z_{k,h}^{\phi^2}$ in the interval $[0, \pi]$ such that:
$$
\sqrt{n^{-2}\lambda_{k,h}^{\phi_1}}= \left(\Psi^{p}_{\phi_1}\right)^{-1}(z_{k,h}^{\phi_1}),
$$
and 
$$
\sqrt{n^{-2}\lambda_{k,h}^{\phi_2}}= \left(\Psi^{p}_{\phi_2}\right)^{-1}(z_{k,h}^{\phi_2}).
$$
With Lemma \ref{approximationwithoutout}, we obtain
\begin{equation}\label{eq:lemma-ditribution-eig-1}
\lim_{n \rightarrow +\infty}\;\max_{k \in \mathcal{I}^{\phi_i}(p,n)}\left|    \sqrt{n^{-2}\lambda_{k,h}^{\phi_i}}   -  \left(\Psi^{p}_{\phi_i}\right)^{-1}\left(\frac{k\pi}{N+1}\right)\right|= 0,\quad i \in \{1,2 \}.
\end{equation}
Using the mean value theorem and Proposition \ref{stefa}, we can further deduce that for $ i \in \{1,2 \}$
\begin{align*}
&\max_{k \in\mathcal{I}^{1,2}(p,n)}\left|  z_{k,h}^{\phi_i} -\frac{k\pi}{N+1} \right |\\
&\leq \left| \sup_{y \geq 0} \left(\Psi^{p}_{\phi_i}\right)'(y)\right| \max_{k \in \mathcal{I}^{1,2}(p,n)}\left|  \sqrt{n^{-2}\lambda_{k,h}^{\phi_i}}   -  \left(\Psi^{p}_{\phi_i}\right)^{-1}\left(\frac{k\pi}{N+1}\right) \right |,
\end{align*}
which concludes the proof of Lemma \ref{theo:lemma-ditribution-eig-1} using \eqref{eq:lemma-ditribution-eig-1}.
\end{proof}
Now, we can proceed to the proof of the theorem. Let $(z_{k,h}^{\phi^1})$ and $(z_{k,h}^{\phi^2})$ be the two sequences constructed in Lemma \ref{theo:lemma-ditribution-eig-1}, and let $\varepsilon = \min_{y\in I} \left( \Psi_{\phi_1}^{p}(\sqrt{y})- \Psi^{p}_{\phi_2}(\sqrt{y})\right)$. Note that the continuity of $y \mapsto \Psi^{p}_{\phi_i}(y)$, $i=1,2$, ensures that $\varepsilon > 0$, hence, using Lemma \ref{theo:lemma-ditribution-eig-1}, there exists $n^0 \in \mathbb{N}^*$, such that:
\begin{equation}\label{eq:lemma-ditribution-eig-2}
\max_{k \in \mathcal{I}^{1,2}(p,n)} \left| z_{k,h}^{\phi_1}-   z_{k,h}^{\phi_2}  \right|< \frac{\varepsilon}{2}, \quad \forall n \geq n^0.
\end{equation}
Now, let $n \geq n^0$ and $k \in \mathbb{N}^*$ such that $n^{-2}\lambda_{k,h}^{\phi_1},n^{-2}\lambda_{k,h}^{\phi_2}\in I$, which implies by definition of $I$ that $k\in \mathcal{I}^{1,2}(p,n)$. To conclude the proof, it is sufficient to prove that
$$
\sqrt{n^{-2}\lambda_{k,h}^{\phi_1}}<\sqrt{n^{-2}\lambda_{k,h}^{\phi_2}}.
$$
In fact, if we suppose the contrary $\left(\sqrt{n^{-2}\lambda_{k,h}^{\phi_1}}\geq \sqrt{n^{-2}\lambda_{k,h}^{\phi_2}}\right)$, we obtain
$$ 
\left| z_{k,h}^{\phi_1}-   z_{k,h}^{\phi_2}  \right|= \left| \Psi_{\phi_1}^{p}\left(\sqrt{n^{-2}\lambda_{k,h}^{\phi_1} }\right)-\Psi^{p}_{\phi_2}\left(\sqrt{n^{-2}\lambda_{k,h}^{\phi_2}} \right)\right|<\frac{\varepsilon}{2},
$$
and using increasing property of $\Psi_{\phi_i}^p$, $i=1,2$ and the condition \eqref{hypoonordeing}, we get
$$
\Psi_{\phi_1}^{p}\left(\sqrt{n^{-2}\lambda_{k,h}^{\phi_1}} \right)\geq \Psi_{\phi_1}^{p}\left(\sqrt{n^{-2}\lambda_{k,h}^{\phi_2}} \right)\geq \Psi^{p}_{\phi_2}\left(\sqrt{n^{-2}\lambda_{k,h}^{\phi_2}} \right), $$ $$\quad \Psi^{p}_{\phi_2}\left(\sqrt{n^{-2}\lambda_{k,h}^{\phi_2}} \right)\leq \Psi^{p}_{\phi_2}\left(\sqrt{n^{-2}\lambda_{k,h}^{\phi_1}} \right).
$$
This yields
$$
\left| z_{k,h}^{\phi_1}-   z_{k,h}^{\phi_2}  \right|\geq \Psi_{\phi_1}^{p}\left(\sqrt{n^{-2}\lambda_{k,h}^{\phi_1}} \right)-\Psi^{p}_{\phi_2}\left(\sqrt{n^{-2}\lambda_{k,h}^{\phi_1}} \right)\geq  \varepsilon, 
$$
which contradicts \eqref{eq:lemma-ditribution-eig-2}. This concludes the proof.
\end{proof}

The previous theorem established a crucial relationship: it demonstrated that the ordering of eigenfrequencies is precisely the inverse of the ordering of the functions $\Psi_{\phi_1}^{p}$ and $\Psi_{\phi_2}^{p}$. However, it is important to note that accessing and manipulating the functions $\Psi_{\phi_1}^{p}$ and $\Psi_{\phi_2}^{p}$ can be complex and difficult in practice. Therefore, there is a need for a more general relationship that connects the ordering of reparametrizations to the order of the associated eigenfrequencies.

The following theorem demonstrates the possibility of exerting control over the distribution of eigenfrequencies through the reparametrization set $\mathbf{C}_{[0,1]}$, offering a valuable tool for managing their behavior more effectively.

.

\begin{theorem} \label{dis1}
Let $\phi_1, \phi_2 \in \mathbf{C}_{[0,1]}$ that are strictly convex, and  $\phi_1'(0) = \phi_2'(0)$. Let $x_0 \in (0, 1)$ to be the first zero of the function $\phi_1' - \phi_2'$. Then, if $$\phi_1'(x) \geq \phi_2'(x),\quad \forall x \in [0, x_0], $$ we have, 
$$ 
\Psi_{\phi_1}^{p}(\sqrt{y}) > \Psi_{\phi_2}^{p}(\sqrt{y}), \quad \forall y \in\left(\frac{e_p(\pi)}{(\phi_1'(x_0))^2}, \frac{e_p(\pi)}{(\phi_1'(0))^2}\right).
$$
\end{theorem}

\begin{proof}
First, we observe that
$$
Rg(\omega^p_{\phi_1})=Rg(\omega^p_{\phi_2})=\left[0,\frac{e_p(\pi)}{(\phi_1^{'}(0))^2}\right],
$$
which implies that the restrictions of $\Psi_{\phi_1}^{p}$ and $\Psi_{\phi_2}^{p}$ on the interval $\left(\frac{\sqrt{e_p(\pi)}}{\phi_1'(x_0)}, \frac{\sqrt{e_p(\pi)}}{\phi_1'(0)}\right)$ are well-defined. On the other hand, Rolle’s theorem ensures the existence of the zero $x_0$ in the open interval $(0,1)$ since $\phi_1(0) - \phi_2(0) = \phi_1(1) - \phi_2(1)$.

Now, let us assume that there exists $y\in\left(\frac{e_p(\pi)}{(\phi_1'(x_0))^2}, \frac{e_p(\pi)}{(\phi_1'(0))^2}\right)$ such that\\ $\Psi_{\phi_1}^{p}(\sqrt{y})\leq \Psi_{\phi_2}^{p}(\sqrt{y})$. From \cite{bianchi2018spectral} (see pages $20$ and $21$), we have 
\begin{equation}\label{eq:therem-dist1-eq1}  
\begin{split}
\Psi_{\phi_1}^{p}(\sqrt{y})&=\pi-\displaystyle\int_{S_1(y)} \left(\phi_1^{'}\right)^{-1}\left(\sqrt{\dfrac{e_p(\theta)}{y}} \right)\;d\theta -\mu_{1}(A_1)
\\
\Psi_{\phi_2}^{p}(\sqrt{y})&=\pi-\displaystyle\int_{S_2(y)}\left(\phi_2^{'}\right)^{-1}\left(\sqrt{\dfrac{e_p(\theta)}{y}} \right)\;d\theta -\mu_{1}(A_2),
\end{split}
\end{equation}
where
$$
S_i(y)=\left\{ \theta\in[0,\pi]:\quad  \phi_i'(0)\leq \sqrt{\dfrac{e_p(\theta)}{y}} \leq\phi_i'(1)    \right\},
$$
and
$$
A_i(y)=\left\{ \theta\in[0,\pi]:\;    \sqrt{\dfrac{e_p(\theta)}{y}}> \phi_i'(1)   \right\}\cup\left\{ \theta\in[0,\pi]:\;    \sqrt{\dfrac{e_p(\theta)}{y}}< \phi_i'(0)   \right\},\;i \in \{1,2\}.
$$
We  claim that
\begin{equation}\label{eq:therem-dist1-eq2}  
S_1(y) = S_2(y) = \left[    e_p^{-1}\left( y\left( \phi_1^{'}(0)   \right)^2 \right),\pi   \right],
\end{equation}
and
\begin{equation}\label{eq:onA}
    A_1=A_2=\left\{ \theta\in[0,\pi]:\;    \sqrt{\dfrac{e_p(\theta)}{y}}< \phi_1'(0)   \right\}.
\end{equation}

Indeed, considering that $y\in \left(\frac{e_p(\pi)}{(\phi_1^{'}(x_0))^2},\frac{e_p(\pi)}{(\phi_1^{'}(0))^2}\right)$ together with $\phi_1'(0)=\phi_2'(0)$ and $\phi_1'(x_0)=\phi_2'(x_0)$, we obtain

$$\phi_i^{'}(0)\leq \sqrt{\frac{e_p(\pi)}{y}}\leq \phi_i^{'}(x_0),\quad i=1,2.$$
Then, using the fact that $e_p(\theta)\leq e_p(\pi)$ for all $\theta\in[0,\pi]$ (see Corollary \ref{eppropreties}), we obtain
$$\sqrt{\frac{e_p(\theta)}{y}}\leq \phi_i^{'}(x_0)\leq \phi_i^{'}(1),\quad \forall \theta\in[0,\pi],\quad i=1,2.$$
We conclude 
$$
S_1(y) = S_2(y) = \left\{ \theta\in[0,\pi]:\quad  \phi_1^{'}(0)\leq \sqrt{\dfrac{e_p(\theta)}{y}} \leq\phi_1^{'}(x_0) \right\},
$$
and 
$$
A_1=A_2=\left\{ \theta\in[0,\pi]:\;    \sqrt{\dfrac{e_p(\theta)}{y}}< \phi_1'(0)   \right\}.
$$
Moreover, since $y(\phi_1'(x_0))^2 > e_p(\pi)$, it is straightforward to see that
\begin{align*}
S_1(y) & = \left\{ \theta\in [0,\pi]: \quad  \phi_1'(0) \leq \sqrt{\dfrac{e_p(\theta)}{y}} \leq\phi_1'(x_0) \right\}\\
& = \left\{ \theta\in [0,\pi]: \quad y\left( \phi_1'(0)   \right)^2 \leq e_p(\theta) \leq y\left( \phi_1'(x_0)   \right)^2 \right\}\\
& = \left[    e_p^{-1}\left( y\left( \phi_1^{'}(0)   \right)^2 \right),\pi   \right],
\end{align*} 
which prove the claim \eqref{eq:therem-dist1-eq2}-\eqref{eq:onA}. On the other hand, the assumption $\Psi_{\phi_1}^{p}(\sqrt{y})\leq \Psi_{\phi_2}^{p}(\sqrt{y})$ implies 
\begin{equation}\label{eq:dif}  \displaystyle\int_{S_1(y)}\left(\phi_2^{'}\right)^{-1}\left(\sqrt{\dfrac{e_p(\theta)}{y}} \right)-\left(\phi_1^{'}\right)^{-1}\left(\sqrt{\dfrac{e_p(\theta)}{y}} \right)\;d\theta\leq 0. 
\end{equation}
Since $0 < y(\phi_1'(0))^2 < e_p(\pi)$, we have $\mu_1\left(S_1(y)\right)>0$. Then, using fact that $\phi_1'(x) \geq \phi_2'(x)$ for all $x \in [0, x_0]$, the above inequality \eqref{eq:dif} gives
\begin{equation}\label{eq:therem-dist1-eq3}
\left(\phi_1'\right)^{-1}\left(\sqrt{\dfrac{e_p(\theta)}{y}} \right) =\left(\phi_2'\right)^{-1}\left(\sqrt{\dfrac{e_p(\theta)}{y}} \right), \quad \forall \theta \in S_1(y).
\end{equation}

To end the proof, it is enough to demonstrate that \eqref{eq:therem-dist1-eq3} contradicts the fact that $x_0 \in (0,1)$ is the first zero of the function $\phi_1'-\phi_2'$. In fact, by the intermediate-value theorem and  \eqref{eq:therem-dist1-eq2}, one can find $x_1 \in (0,x_0)$ and $\theta_0 \in S_1(y)$ such that
$$
\phi_1'(x_1) = \sqrt{\frac{e_p(\theta_0)}{y}},
$$
and, using \eqref{eq:therem-dist1-eq3}, we obtain 
$$
x_1=\left(\phi_1'\right)^{-1}\left( \sqrt{\frac{e_p(\theta_0)}{y}}\right) = \left(\phi_2'\right)^{-1}\left( \sqrt{\frac{e_p(\theta_0)}{y}}\right) = \left(\phi_2'\right)^{-1}\left(\phi_1'(x_1)\right).
$$
Consequently, $\phi_2'(x_1) = \phi_1'(x_1)$, signifying that $x_1$ is a zero for the function $\phi_1'-\phi_2'$. This concludes the proof as $x_1<x_0$.
\end{proof}

 Similarly, the next theorem addresses strictly concave reparametrizations. The proof follows a similar approach to that of Theorem \ref{dis1} and is thus omitted.

\begin{theorem}\label{dis2}
Let $\phi_1, \phi_2 \in \mathbf{C}_{[0,1]}$ that are strictly concave, and  $\phi_1'(1) = \phi_2'(1)$. Let $x_0 \in (0, 1)$ to be the last zero of the function $\phi_1' - \phi_2'$. Then, if 
$$\phi_1'(x) \geq \phi_2'(x),\quad \forall x \in [x_0, 1],$$
we have, 
$$ 
\Psi_{\phi_1}^{p}(\sqrt{y}) > \Psi_{\phi_2}^{p}(\sqrt{y}), \quad \forall y \in \left(\frac{e_p(\pi)}{(\phi_1'(x_0))^2}, \frac{e_p(\pi)}{(\phi_1'(1))^2}\right).
$$
\end{theorem}

A family of examples of such convex reparametrization functions is 
\begin{equation}\label{eq:reparametrization-convex-example}
\phi_{a,b}=e^{ax+b}-e^b+(\gamma-ae^b)x,
\end{equation}
where $a >0$, $0<\gamma<1$ and $b$ is given by 
$$
b=-\ln\left(\frac{e^a-(a+1)}{1-\gamma}\right).
$$
Notably, these reparametrization functions satisfy the conditions specified in Theorem \ref{dis1}, including the property $\phi_{a,b}'(0) = \gamma$.  

In the case of strictly concave reparametrization, the family \eqref{eq:reparametrization-convex-example} is replaced by
\begin{equation}\label{eq:reparametrization-concave-example}
\phi_{a,b}=\ln(ax+b)-\ln(b)+\left(\gamma-\frac{a}{a+b}\right)x.
\end{equation}
Here, the parameters are similar: $a > 0$, $0 < \gamma < 1$. However, the value of $b$ is calculated as $b:=a/x^*$, where $x^* \in (0,1)$ represents the unique solution to the equation
$$
\gamma=1-\left(\ln(x^*+1)-\frac{x^*}{x^*+1}\right).
$$ 
Once again, these reparametrization functions satisfy the hypotheses of Theorem \ref{dis2}, and they specifically verify $\phi_{a,b}'(1)=\gamma$.

At this stage, we have demonstrated the feasibility of ordering the families of eigenfrequencies based on the corresponding reparametrizations' ordering. In the following result, our focus will be on the distribution of packed eigenfrequencies and how the convexity of the symbol influences the behavior of these eigenfrequencies. More precisely, we will show that the number of eigenfrequencies can indeed be ordered, and this order depends on the convexity of the symbol. In the case of a strictly convex symbol, the frequency spectrum shifts to the right, whereas a strictly concave symbol results in a leftward shift.

\begin{theorem}[Distribution of pack-eigenvalues]\label{cl2.5}
Let $p \geq 1$ and $\phi \in \mathbf{C}_{[0,1]}$. Suppose there exists a closed interval $I$ in which $\Psi^p_{\phi} $ is of class $\mathcal{C}^2(I)$. Let $\{y_0, y_1, \ldots, y_r\}$ be a uniform discretization of $I$ and define $E_I^{i}$ by 
$$E_I^{i}=\left|\left\{ k=1,\dots,N: y_{i-1}< \sqrt{n^{-2}  \lambda_{k,h}^{\phi}}\leq y_{i}     \right\}\right|,\quad i=1,\dots r-1.$$
Then, for sufficiently large $n$, the following hold
\begin{itemize}
\item[1.] If $\Psi^p_{\phi}$ is strictly convex on $I$, then for all $i \in \{1, 2, \ldots, r-1\}:$
\begin{equation*}
E_I^{i}<E_I^{i+1}.
\end{equation*}

\item[2.] If $\Psi^p_{\phi}$ is strictly concave on $I$, then for all $i \in \{1, 2, \ldots, r-1\}:$
\begin{equation*}
E_I^{i}> E_I^{i+1}.
\end{equation*}
\end{itemize}
\end{theorem}

\begin{proof}
We will focus on proving assertion 1, as the proof of assertion 2 follows a similar argument. For all $i \in \{1, \cdots, r\}$, we apply Discrete Weyl's law \ref{th1} to obtain
\begin{equation} \label{eq:propo-distribution-pack-eigenvalues-1}
\lim_{n\rightarrow  +\infty}\dfrac{E_I^{i}}{n}=\frac{1}{\pi}  \Psi^p_{\phi}(y_{i})-\frac{1}{\pi}  \Psi^p_{\phi}(y_{i-1}),
\end{equation}
and by the mean value theorem, there exist $x_i \in (y_{i-1}, y_i)$ such that
\begin{equation} \label{eq:propo-distribution-pack-eigenvalues-2}
\begin{split}
\frac{1}{\pi}  \Psi^p_{\phi}(y_{i})-\frac{1}{\pi}  \Psi^p_{\phi}(y_{i-1}) &= \frac{1}{\pi}(y_{i}-y_{i-1}) \left(  \Psi^p_{\phi}\right)'(x_i)\\
&= \frac{y_r-y_0}{r\pi}\left(  \Psi^p_{\phi}\right)'(x_i).
\end{split}
\end{equation}
Now, for $i \in \{1, \cdots, r-1\}$, let us fix $\epsilon$ such that
$$
0 < \epsilon < \frac{(y_r-y_0) \delta}{2r\pi} \inf_{I} \left(  \Psi^p_{\phi}   \right)'', \quad \text{with } \delta:= \inf_{1 \leq i \leq r-1} (x_{i+1} -x_{i} ). 
$$
Using equations \eqref{eq:propo-distribution-pack-eigenvalues-1} and \eqref{eq:propo-distribution-pack-eigenvalues-2}, we deduce that for sufficiently large $n$:
\begin{align*}
& \dfrac{1}{n}\left(E^i_{I} -E^{i+1}_I       \right) \\
&\leq 2\epsilon +\frac{1}{\pi}\left[    \Psi^p_{\phi}(y_{i})-  \Psi^p_{\phi}(y_{i-1})  \right] - \frac{1}{\pi} \left[   \Psi^p_{\phi}(y_{i+1})-  \Psi^p_{\phi}(y_{i})\right]\\
&=2\epsilon+ \frac{y_r-y_0}{r\pi}\left(   \left(  \Psi^p_{\phi}\right)^{'}(x_i) -\left(  \Psi^p_{\phi}\right)^{'}(x_{i+1})         \right)\\
&\leq 2\epsilon -\frac{y_r-y_0}{r\pi}\inf_{I} \left(  \Psi^p_{\phi}\right)'' \;(x_{i+1}-x_i)\\
&\leq 2\epsilon -\frac{(y_r-y_0) \delta}{r\pi} \inf_{I} \left(  \Psi^p_{\phi}\right)'' <0.
\end{align*}
This concludes the proof.
\end{proof}
We end this section with the following remark
\begin{remark}
Proposition \ref{propc} provides all the necessary conditions on the function $\Psi^p_{\phi}$ in the case $p=1$ to apply the above theorem. Although the results of the theorem are general, a study of the symbol's regularity and convexity is needed to understand the shifting of the eigenfrequency spectrum in the general case $p \geq 1$ and within the entire range of $\sqrt{\omega^p_{\phi}}$.
\end{remark}

\section{Improved eigenfrequency estimations }\label{improvedestimation}
In this section, we leverage Proposition \ref{stefa} and the linear behavior of the symbol near zero to improve convergence in both the discrete Weyl's law \eqref{discretweylslaw} and the eigenfrequency estimation \eqref{approximationwithoutout}. Specifically, we use the properties of $\Psi^p_{\phi}$ from Proposition \ref{stefa} to establish that the simple convergence in \eqref{conversdiscretwels} is, in fact, uniform. Additionally, we build on Theorem \ref{symbolnear0} to derive new estimates for the eigenfrequencies.

The first result of this section reflects that the normalized count of eigenfrequencies less than a given value $y$ in $Rg\left(\sqrt{\omega_\phi^p}\right)$ converges uniformly to $\frac{1}{\pi}\Psi^p_{\phi}(y)$.
\begin{theorem}[Uniform Discrete Weyl’s law]\label{gedis}
Let $\phi\in\mathbf{C}_{[0,1]}$, such that \\$\sqrt{n^{-2}L_{\phi,n}^p}\sim_{\lambda} \sqrt{ \xi^p_{\phi}}$. Then the sequence of  functions 
$$
G_n^p(y)= \dfrac{\left| \left\{k=1,\dots, N:\quad \sqrt{n^{-2}\lambda_{k,h}^{\phi}}\leq y\right\}\right|}{N+1},\quad \forall y\in Rg\left(\sqrt{\omega_{\phi}^p}\right),
$$
converges uniformly, as $n \rightarrow +\infty$, to $\frac{1}{\pi}\Psi^p_{\phi}$.
\end{theorem}

\begin{proof}
Let $n \in \mathbb{N}^*$. We observe that the range of the function $\sqrt{\omega_\phi^p}$ is decomposed as 
\begin{equation*}
\begin{split}
    Rg\left(\sqrt{\omega_\phi^p}\right)&=\left[0,\sqrt{n^{-2}\lambda_{1,h}^{\phi}}\right)\cup\left(\bigcup_{k=1}^{N-OUT(p,n)-1}\left[\sqrt{n^{-2}\lambda_{k,h}^{\phi}},\sqrt{n^{-2}\lambda_{k+1,h}^{\phi}}\right)\right)\\&\hspace{3cm}\cup\left[\sqrt{n^{-2}\lambda_{N-OUT(p,n),h}^{\phi}},\max Rg\left(\sqrt{\omega_\phi^p}\right)\right].      
\end{split}
\end{equation*}
Hence,
\begin{equation}\label{eq:gedis-1}
\max_{y\in Rg\left(\sqrt{\omega_\phi^p}\right)}\left| G_n^p(y) - \frac{1}{\pi}\Psi^p_{\phi}(y)    \right|=\max \left \{T_1,T_2,T_3 \right\},
\end{equation}
where 

\begin{align*}
& T_1=\max\left\{\frac{1}{\pi}\Psi^p_{\phi}(y): \quad 0\leq y<  \sqrt{n^{-2}\lambda_{1,h}^{\phi} } \right\},\\
&{}\\
& T_2=\max \left\{ \left|\dfrac{k}{N+1}-  \frac{1}{\pi}\Psi^p_{\phi}(y)  \right|:  1 \leq k \leq N-OUT(p,n)-1 \text{ and } \sqrt{n^{-2}\lambda_{k,h}^{\phi}}\leq y < \sqrt{n^{-2}\lambda_{k+1,h}^{\phi}} \right\},\\
&{and}\\
& T_3=\max\left\{\left|\dfrac{N-OUT(p,n)}{N+1} -\frac{1}{\pi}\Psi^p_{\phi}(y)  \right| :\; \sqrt{n^{-2}\lambda_{N-OUT(p,n),h}^{\phi}}\leq y\leq \max Rg\left(\sqrt{\omega_\phi^p}\right)\right\}.
\end{align*}

Now, we will estimate the three terms $T_1$, $T_2$, and $T_3$. For $T_1$, using the monotonicity of the function $y \mapsto \Psi^p_{\phi}(y)$, we have:
\begin{equation}\label{eq:gedis-2}
T_1\leq \frac{1}{\pi}\Psi^p_{\phi} \left( \sqrt{n^{-2}\lambda_{1,h}^{\phi}} \right).
\end{equation}
Applying the mean value theorem, Proposition \ref{stefa} and the relation \eqref{relatioxiandPsi}, we obtain:
\begin{equation}\label{eq:gedis-3}
T_2\leq \frac{C}{\pi} \max_{1\leq k\leq  N-1-OUT(p,n)}\left(T_{2,k}\right),
\end{equation}
and
\begin{equation}\label{eq:gedis-4}
T_3\leq  \frac{C}{\pi} T_{3,k},
\end{equation}
where $$C:=\sup_{y\geq 0}\left(\Psi^p_{\phi}\right)'(y),$$
$$T_{2,k}=\left(\left|\sqrt{\xi_{\phi}^p}\left(\dfrac{k}{N+1}\right)- \sqrt{n^{-2}\lambda_{k,h}^{\phi}}   \right|+ \left|  \sqrt{ n^{-2}\lambda_{k+1,h}^{\phi}}-\sqrt{n^{-2}\lambda_{k,h}^{\phi}}\right|\right),$$
and 
\begin{equation*}
    \begin{split}
      T_{3,k}&=\left|\sqrt{\xi_{\phi}^p}\left(\dfrac{N-OUT(p,n)}{N+1}\right)- \sqrt{n^{-2}\lambda_{N-OUT(p,n),h}^{\phi}}   \right|\\&\hspace{5cm}+ \left|  \sqrt{ n^{-2}\lambda_{N-OUT(p,n),h}^{\phi}}-\max Rg\left(\sqrt{\omega_\phi^p}\right)\right|.  
    \end{split}
\end{equation*}
We now establish that all the terms $T_1$, $T_2$, and $T_3$ tend to zero as $n$ tends to infinity.

From \eqref{eq:lemma-ditribution-eig-out}, we have 
$$\lim_{n \rightarrow +\infty}\left|    \sqrt{n^{-2}\lambda_{1,h}^{\phi}}   -  \left(\Psi^{p}_{\phi}\right)^{-1}\left(\frac{\pi}{N+1}\right)\right|= 0.$$
Since the function $ \left(\Psi^{p}_{\phi}\right)^{-1}$ is continuous, then  $$\lim_{n \rightarrow +\infty} \left(\Psi^{p}_{\phi}\right)^{-1}\left(\frac{\pi}{N+1}\right)=\left(\Psi^{p}_{\phi}\right)^{-1}\left(0\right)=0. $$
Hence,
\begin{equation}\label{limitfirsfreq}
    \lim_{n \rightarrow +\infty} \sqrt{n^{-2}\lambda_{1,h}^{\phi}} =0.
\end{equation}
Injecting  the above limit \eqref{limitfirsfreq} in \eqref{eq:gedis-2} and using the fact that $\Psi^p_{\phi}(0)=0,
$
we get 
\begin{equation}\label{limitT1}
   \lim_{n \rightarrow +\infty}T_1=0 .
\end{equation}
For the term $T_2,$ 
we have 
\begin{align*}
    T_{2,k}&=\left|\sqrt{\xi_{\phi}^p}\left(\dfrac{k}{N+1}\right)- \sqrt{n^{-2}\lambda_{k,h}^{\phi}}   \right|+ \left|  \sqrt{ n^{-2}\lambda_{k+1,h}^{\phi}}-\sqrt{n^{-2}\lambda_{k,h}^{\phi}}\right|\\
    &\leq 2 \left|\sqrt{\xi_{\phi}^p}\left(\dfrac{k}{N+1}\right)- \sqrt{n^{-2}\lambda_{k,h}^{\phi}}   \right| + \left|\sqrt{\xi_{\phi}^p}\left(\dfrac{k+1}{N+1}\right)- \sqrt{n^{-2}\lambda_{k+1,h}^{\phi}}   \right|\\
    &\quad\quad\quad\quad\quad\quad+\left| \sqrt{\xi_{\phi}^p}\left(\dfrac{k+1}{N+1}\right)-\sqrt{\xi_{\phi}^p}\left(\dfrac{k}{N+1}\right) \right|.
\end{align*}
Then,
\begin{align}\label{divlt2}
    T_2\leq \frac{C}{\pi}\Bigg\{3 &\max_{k \in \mathcal{I}(p,n)}\left|\sqrt{\xi_{\phi}^p}\left(\dfrac{k}{N+1}\right)- \sqrt{n^{-2}\lambda_{k,h}^{\phi}}   \right| \nonumber \\
    &+ \max_{k=1,\cdots N-1-OUT(p,n)} \left| \sqrt{\xi_{\phi}^p}\left(\dfrac{k+1}{N+1}\right)-\sqrt{\xi_{\phi}^p}\left(\dfrac{k}{N+1}\right) \right|\Bigg\}.
\end{align}
On the other hand, since the function $\sqrt{\xi_{\phi}^p}$ is continuous over  $[0,1]$, it is also uniformly continuous. This implies that
\begin{equation}\label{uniformconverxi}
    \lim_{n \rightarrow +\infty} \max_{k=1,\cdots N-1-OUT(p,n)} \left| \sqrt{\xi_{\phi}^p}\left(\dfrac{k+1}{N+1}\right)-\sqrt{\xi_{\phi}^p}\left(\dfrac{k}{N+1}\right) \right|=0.
\end{equation}
Inserting the limit \eqref{uniformconverxi} along with the estimate \eqref{eq:lemma-ditribution-eig-out} into inequality \eqref{divlt2} leads to 
\begin{equation}\label{limitT2}
    \lim_{n \rightarrow +\infty} T_2=0.
\end{equation}
We now, check the limit of the last term $T_3$. Using \eqref{eq:lemma-ditribution-eig-out}, we obtain
\begin{equation}\label{limitT21}
  \lim_{n \rightarrow +\infty}  \left|\sqrt{\xi_{\phi}^p}\left(\dfrac{N-OUT(p,n)}{N+1}\right)- \sqrt{n^{-2}\lambda_{N-OUT(p,n),h}^{\phi}}   \right|=0.
\end{equation}
From Corollary \ref{numberofoutlierscor}, we have $OUT(p,n)=o(N)$. Then 
\begin{equation}\label{limitoutlier}
   \lim_{n \rightarrow +\infty}\frac{N-OUT(p,n)}{N+1}=1. 
\end{equation}
Injecting \eqref{limitoutlier} in \eqref{limitT21}  and employing the continuity of $\sqrt{\xi^p_{\phi}}$, we get
$$\lim_{n \rightarrow +\infty} \sqrt{n^{-2}\lambda_{N-OUT(p,n),h}^{\phi}} = \sqrt{\xi^p_{\phi}}(1).$$
Observe that $\max Rg\left(\sqrt{\omega^p_{\phi}}\right)=\sqrt{\xi^p_{\phi}}(1)$. Then 
\begin{equation}\label{T3limit2}
  \lim_{n \rightarrow +\infty}  \left|  \sqrt{ n^{-2}\lambda_{N-OUT(p,n),h}^{\phi}}-\max Rg\left(\sqrt{\omega_\phi^p}\right)\right|=0.
\end{equation}
Taking the limit in \eqref{eq:gedis-4} and using the two limits \eqref{T3limit2} and  \eqref{limitT21}, we obtain 
\begin{equation}\label{limitT3fn}
    \lim_{n \rightarrow +\infty} T_3=0.
\end{equation}
Applying the limit in \eqref{eq:gedis-1} and incorporating the three derived limits \eqref{limitT1}, \eqref{limitT2}, and \eqref{limitT3fn} concludes the proof.

\end{proof}

In the following, we leverage the symbol's asymptotic behavior in the zero neighborhood to derive new estimates for the eigenfrequencies. The next result indicates an  equivalence between the eigenfrequency behavior near zero and the order of approximation in \eqref{eq:lemma-ditribution-eig-out}.
\begin{corollary}\label{equivalenceimportant}
    Let $\phi\in\mathbf{C}_{[0,1]}$, such that $\sqrt{n^{-2}L_{\phi,n}^p}\sim_{\lambda} \sqrt{\xi^p_{\phi}}$. Suppose there exists a constant $\beta\in \mathbb{R}$, such that for any subsequence $(k(n))\subseteq\mathcal{I}(p,n) $ satisfying  $\lim_{n \rightarrow +\infty}\frac{k(n)}{N+1}=0$, the following holds:  
    \begin{equation}\label{assumtiononeigenvaluesnear0}
        \lim_{n \rightarrow +\infty}\frac{\sqrt{n^{-2}\lambda_{k(n),h}^{\phi}}}{\gamma \frac{k(n)}{N+1}}=\beta.
    \end{equation}
    where $\gamma$ is given by \eqref{gammadif}. Then, the  following two statements are equivalent:
  \begin{enumerate}
    \item[(i)] $\beta=1.$
    \item[(ii)] $ \lim_{n \rightarrow +\infty}\;\max_{k \in \mathcal{I}(p,n)} \left\{\frac{N+1}{k}\left|\sqrt{n^{-2}\lambda_{k,h}^{\phi}}-\sqrt{\xi_{\phi}^p}\left(\frac{k}{N+1}\right)\right|\right\}=0.$
\end{enumerate}
\end{corollary}
\begin{proof}
Assuming that (\ref{assumtiononeigenvaluesnear0}) holds, we proceed with the first implication, showing that $(i)$ implies $(ii)$ using proof by contradiction. Specifically, we assume $(i)$ holds and then suppose, for contradiction, that $(ii)$ is false. Then,  there exists a subsequence $((k(n))\subset \mathcal{I}(p,n)$, such that 
 \begin{equation}\label{contadict1}
     \frac{N+1}{k(n)}\left|\sqrt{n^{-2}\lambda_{k(n),h}^{\phi}}-\sqrt{\xi_{\phi}^p}\left(\frac{k(n)}{N+1}\right)\right|>\varepsilon,
 \end{equation}
 for some $\varepsilon>0$  independent of $n$.\\

 Since  $\left(\frac{k(n)}{N+1}\right)$ is bounded, there exists a subsequence, which we also denote by  $\left(\frac{k(n)}{N+1}\right)$, such that (\ref{contadict1}) holds and $$\lim_{n \rightarrow +\infty}\frac{k(n)}{N+1}=x_0\in[0,1].$$
By the continuity of $\sqrt{\xi^p_{\phi}}$ and (\ref{assymptoticatzero}), we obtain
 \begin{equation}\label{limitesofsymbol}
     \lim_{n \rightarrow +\infty}\frac{\sqrt{\xi_{\phi}^p}\left(\frac{k(n)}{N+1}\right)}{\frac{k(n)}{N+1}}=\left\{
\begin{array}{ll}
\frac{\sqrt{\xi_{\phi}^p}(x_0)}{x_0},&\quad x_0\neq 0,\\
\gamma,&\quad x_0=0.
\end{array}
\right.\end{equation}
Using the assumption (\ref{assumtiononeigenvaluesnear0}) along with the Discrete Weyl’s law (\ref{discretweylslaw}), we have
\begin{equation}\label{limitesoffreq}
     \lim_{n \rightarrow +\infty}\frac{\sqrt{n^{-2}\lambda_{k(n),h}^{\phi}}}{\frac{k(n)}{N+1}}=\left\{
\begin{array}{ll}
\frac{\sqrt{\xi_{\phi}^p}(x_0)}{x_0},&\quad x_0\neq 0,\\
\gamma,&\quad x_0=0.
\end{array}
\right.\end{equation}

By substituting  (\ref{limitesofsymbol}) and (\ref{limitesoffreq}) in (\ref{contadict1}),  we arrive at $0\geq \varepsilon$, which is a contradiction. This completes the proof of the first implication.

We now focus on establishing the second implication, namely, that $(ii)$ implies $(i)$. Let $(k(n))\subset \mathcal{I}(p,n)$, such that $\lim_{n \rightarrow +\infty}\frac{k(n)}{N+1}=0$. From $(ii)$, we have 
\begin{equation}\label{contadict2}
    \lim_{n \rightarrow +\infty} \frac{N+1}{k(n)}\left|\sqrt{n^{-2}\lambda_{k(n),h}^{\phi}}-\sqrt{\xi_{\phi}^p}\left(\frac{k(n)}{N+1}\right)\right|=0.
 \end{equation}
 On the other hand, 
 \begin{equation}\label{cond3}
     \frac{\sqrt{n^{-2}\lambda_{k(n),h}^{\phi}}}{\gamma \frac{k(n)}{N+1}}=\frac{1}{\gamma}\left[ \frac{\sqrt{n^{-2}\lambda_{k(n),h}^{\phi}}}{\frac{k(n)}{N+1}}-\frac{\sqrt{\xi_{\phi}^p}\left(\frac{k(n)}{N+1}\right)}{\frac{k(n)}{N+1}}\right]+\frac{1}{\gamma}\frac{\sqrt{\xi_{\phi}^p}\left(\frac{k(n)}{N+1}\right)}{\frac{k(n)}{N+1}}.
 \end{equation}
By taking the limit in the equality (\ref{cond3}) and using (\ref{contadict2}) along with the linear behavior of the symbol near zero (\ref{assymptoticatzero}), we obtain $\beta=1$, which concludes the proof.
\end{proof}
The next result gives an equivalence estimate to (\ref{eq:lemma-ditribution-eig-out}), illustrating the asymptotic equivalence between the sequence of eigenfrequencies and the sequence generated by the uniform sampling of the symbol, up to an additional sequence term.

\begin{corollary}\label{improvedfrequincyestimation}
    Let $\phi\in\mathbf{C}_{[0,1]}$, such that  $\sqrt{n^{-2}L_{\phi,n}^p}\sim_{\lambda} \sqrt{\xi^p_{\phi}}$. Then the estimation (\ref{eq:lemma-ditribution-eig-out}) is equivalent to 
    \begin{equation}\label{improved1}
        \lim_{n \rightarrow +\infty}\;\max_{k \in \mathcal{I}(p,n)} \left\{\frac{k}{N+1}\left|\frac{\sqrt{n^{-2}\lambda^{\phi}_{k,h}}}{\sqrt{\xi^p_{\phi}}\left(\frac{k}{N+1}\right)}-1 \right|\right\}=0.
    \end{equation}
\end{corollary}
\begin{proof}
    We will establish only the first implication, namely, that (\ref{eq:lemma-ditribution-eig-out}) implies (\ref{improved1}). The second implication can be derived using similar steps. Let $n\in\mathbb{N}^{\ast}$, for all $k\in \mathcal{I}(p,n)$, we have 
    \begin{equation}\label{cond4}
        \frac{k}{N+1}\left|\frac{\sqrt{n^{-2}\lambda^{\phi}_{k,h}}}{\sqrt{\xi^p_{\phi}}\left(\frac{k}{N+1}\right)}-1 \right|=\frac{\frac{k}{N+1}}{\sqrt{\xi^p_{\phi}}\left(\frac{k}{N+1}\right)}\left| \sqrt{n^{-2}\lambda^{\phi}_{k,h}}-\sqrt{\xi^p_{\phi}}\left(\frac{k}{N+1}\right)\right|.
    \end{equation}

Now, assume for contradiction that (\ref{improved1}) is false. Then, following an analysis similar to that used in the proof of $(i)$ implies $(ii)$ in the above  corollary, we can construct a subsequence  $(k(n))\subset  \mathcal{I}(p,n)$, such that 
$$\lim_{n \rightarrow +\infty}\frac{k(n)}{N+1}=x_0\in[0,1]\;\;\text{and}\;\;\frac{k(n)}{N+1}\left|\frac{\sqrt{n^{-2}\lambda^{\phi}_{k(n),h}}}{\sqrt{\xi^p_{\phi}}\left(\frac{k(n)}{N+1}\right)}-1 \right|>\varepsilon>0,$$
for some $\varepsilon$  independent of $n$. Then, using  (\ref{cond4}), we obtain
$$\left| \sqrt{n^{-2}\lambda^{\phi}_{k(n),h}}-\sqrt{\xi^p_{\phi}}\left(\frac{k(n)}{N+1}\right)\right|>\varepsilon\frac{\sqrt{\xi^p_{\phi}}\left(\frac{k(n)}{N+1}\right)}{\frac{k(n)}{N+1}}.$$
By passing to the limit and using (\ref{eq:lemma-ditribution-eig-out}) and (\ref{limitesoffreq}), we get

$$
\left\{
\begin{array}{ll}
0\geq \varepsilon\frac{\sqrt{\xi^p_{\phi}(x_0)}}{x_0},& \text{if } x_0\neq 0\\
0\geq\varepsilon \gamma,& \text{otherwise, } x_0=0.
\end{array}
\right.
$$
 This results in a contradiction, concluding the proof of the first implication.

For the second implication, one can deduce that for all $k\in \mathcal{I}(p,n)$, we have
$$
\left| \sqrt{n^{-2}\lambda^{\phi}_{k,h}}-\sqrt{\xi^p_{\phi}}\left(\frac{k}{N+1}\right)\right|=\frac{\sqrt{\xi^p_{\phi}}\left(\frac{k}{N+1}\right)}{\frac{k}{N+1}}\left\{\frac{k}{N+1}\left|\frac{\sqrt{n^{-2}\lambda^{\phi}_{k,h}}}{\sqrt{\xi^p_{\phi}}\left(\frac{k}{N+1}\right)}-1 \right|\right\},
$$
and then apply the same techniques as above to establish that  (\ref{improved1}) implies (\ref{eq:lemma-ditribution-eig-out}), which ends the proof. 
\end{proof}

To conclude this section, we demonstrate that the term $(k/N+1)$ in the above estimate \eqref{improved1} can be eliminated if either condition $(i)$ or $(ii)$ in Corollary \ref{equivalenceimportant} is satisfied.
\begin{corollary}\label{improvedf}
Let $\phi\in\mathbf{C}_{[0,1]}$, such that  $\sqrt{n^{-2}L_{\phi,n}^p}\sim_{\lambda} \sqrt{\xi^p_{\phi}}$. Assume that $(i)$ in Corollary \ref{equivalenceimportant} holds. Then,
 \begin{equation}\label{improved2}
        \lim_{n \rightarrow +\infty}\;\max_{k \in \mathcal{I}(p,n)} \left|\frac{\sqrt{n^{-2}\lambda^{\phi}_{k,h}}}{\sqrt{\xi^p_{\phi}}\left(\frac{k}{N+1}\right)}-1 \right|=0.
    \end{equation}
\end{corollary}
\begin{proof}
The proof follows similar steps as in the proofs of Corollaries \ref{equivalenceimportant} and \ref{improvedfrequincyestimation}. We assume that \eqref{improved2} is false. Then, we use assumption $(i)$ in Corollary \ref{equivalenceimportant}, along with Theorem \ref{symbolnear0}, to handle the approximation when a subsequence $(k(n)/N+1)$ tends to zero. For brevity, the details are omitted.
\end{proof}

\section{Conclusions, further comments, and future work}\label{sec:conclusions}
We have shown that GLT theory can provide complex and novel insights into the distribution of eigenfrequencies and the effect of reparametrization on the eigenfrequency spectrum. Using the regularity of $\Psi^p_{\phi}$ established in \cite{bianchi2018spectral}, we developed an ordering relation among different mappings and their associated families of eigenfrequencies. Furthermore, we showed that this regularity leads to uniform convergence for the discrete Weyl’s law (Theorem \ref{th1}). Additionally, we illustrated how the convexity of $\Psi^p_{\phi}$ can be leveraged to shift the eigenfrequency. Moreover, we established a linear asymptotic behavior of the symbol $\sqrt{\xi^p_{\phi}}$ near zero, enabling us to derive new estimations for eigenfrequencies.

The analysis presented in this study is related to the IGA discretization of the Laplace operator, utilizing regular B-splines. Throughout, our investigation of eigenfrequency distribution reveals that our analysis is independent of the specific discretization method and the particular operator under consideration. What remains crucial is the matrix symbol resulting from a numerical discretization of any given operator, and in this direction, the GLT theory is very flexible,  as already emphasized in the seminal paper [\cite{capizzano2003generalized}, pp. 376-377, Remark 2.1, Remark 2.2, Section 6, Theorem 6.2, Corollary 6.3, Corollary 6.4, Remark 6.4]. Consequently, we can extend the same study to finite difference discretization, higher-order Lagrangian finite element approximation, and optimal spline subspaces \cite{manni2022application}.

To derive the estimations in both Corollary \ref{equivalenceimportant} and Corollary \ref{improvedf}, we assumed that the eigenfrequencies exhibit the same behavior as the symbol near zero, which can be seen as a hypothesis equivalent to assumption (i) in [Theorem 3.3, \cite{bianchi2021analysis}]. In future work, we aim to investigate and formally establish this assumption. Furthermore, in Theorem \ref{tdis1}, we have avoided the intersection points of the functions $\Psi^p_{\phi_1}$ and $\Psi^p_{\phi_2}$; studying the behavior of eigenfrequencies at these points would be interesting. Moreover, in the alternative theorems to Theorem \ref{tdis1}, namely Theorems \ref{dis1} and \ref{dis2}, we have provided only a subset of reparametrizations that allow ordering of eigenfrequencies within a portion of $Rg\left(\sqrt{\omega^p_{\phi}}\right)$.  Further investigation into the eigenfrequency behavior in the rest of the range, along with the construction of a reparametrizations subset that permits ordering across the entire range, remains necessary.

Lastly, it is worth noting that all the findings presented in this paper, which encompass the improvement of the results in \cite{bianchi2021analysis}, along with our analysis of eigenfrequencies distribution, can be extended to domains of higher dimensions (see \cite{serra2006glt}, \cite{barbarino2020block},  and \cite{garoni2017generalized2}).

\section*{Acknowledgments}
The authors are very grateful to Professor Stefano Serra Capizzano for dedicating his time and showing interest in this work. We deeply appreciate the insightful discussions and his invaluable remarks and corrections, which greatly improved the quality of the present paper.

\end{document}